\pdfoutput=1
\documentclass[11pt]{article}

\usepackage[a4paper,margin=28mm]{geometry}

\usepackage[T1]{fontenc}
\usepackage[utf8]{inputenc}
\usepackage{lmodern}
\usepackage{amsmath,amssymb,amsthm,mathtools}
\usepackage{bm}
\usepackage{mathrsfs}

\usepackage{graphicx}
\usepackage{booktabs}
\usepackage{array}
\usepackage{enumitem}

\usepackage{aliascnt}
\usepackage[colorlinks=true,linkcolor=blue,citecolor=blue,urlcolor=blue]{hyperref}
\usepackage[nameinlink,capitalize]{cleveref}

\theoremstyle{plain}
\newtheorem{theorem}{Theorem}[section]

\newaliascnt{lemma}{theorem}
\newtheorem{lemma}[lemma]{Lemma}
\aliascntresetthe{lemma}

\newaliascnt{proposition}{theorem}
\newtheorem{proposition}[proposition]{Proposition}
\aliascntresetthe{proposition}

\newaliascnt{corollary}{theorem}
\newtheorem{corollary}[corollary]{Corollary}
\aliascntresetthe{corollary}

\theoremstyle{definition}

\newaliascnt{definition}{theorem}
\newtheorem{definition}[definition]{Definition}
\aliascntresetthe{definition}

\newaliascnt{assumption}{theorem}

\aliascntresetthe{assumption}

\newaliascnt{example}{theorem}
\newtheorem{example}[example]{Example}
\aliascntresetthe{example}

\theoremstyle{remark}

\newaliascnt{remark}{theorem}
\newtheorem{remark}[remark]{Remark}
\aliascntresetthe{remark}

\crefname{assumption}{Assumption}{Assumptions}
\Crefname{assumption}{Assumption}{Assumptions}

\newcommand{\R}{\mathbb{R}}

\newcommand{\N}{\mathbb{N}}

\newcommand{\Tend}{\mathfrak{T}}          

\newcommand{\Wsp}{\mathcal{W}}            
\newcommand{\norm}[1]{\left\lVert #1 \right\rVert}
\newcommand{\seminorm}[1]{\left\lvert #1 \right\rvert}
\newcommand{\inner}[2]{\left( #1,#2 \right)}
\newcommand{\dual}[2]{\left\langle #1,#2 \right\rangle}
\newcommand{\diff}{\mathop{}\!\mathrm{d}}
\newcommand{\Km}{K_{\mu}}
\newcommand{\Dm}{\mathcal{D}_{\mu}}
\newcommand{\Lp}{L^{2}(\Omega)^{d}}

\makeatletter

\@addtoreset{equation}{section}
\makeatother
\title{Graph-space well-posedness for diffusion equations with degenerate instantaneous diffusion}

\author{
  Hiroki Ishizaka\\
  Team FEM, Matsuyama, Japan\\
  E-mail: \texttt{h.ishizaka005@gmail.com}
}

\date{}

\begin{document}

\maketitle

\begin{abstract}
We study diffusion equations with completely monotone memory when the instantaneous diffusion form is merely non-negative and may therefore lose coercivity.  For a kernel whose Bernstein representing measure has finite total mass
\(M_{0}=\nu([0,\infty))\), we introduce an extended state consisting of the physical variable and its continuum of internal variables.  The aggregation and constant-embedding operators are adjoint with respect to the memory energy, and the resulting cross-term cancellation makes the augmented generator $m$-dissipative.  This yields a unique mild solution, Lipschitz dependence on the data, and a contraction estimate that contains no positive lower bound for the instantaneous form.  The zero-prehistory trajectories form a memory graph space, in which the problem is well posed in the sense of Hadamard.  If, in addition, the first Bernstein moment
\(M_{1}=\int_{[0,\infty)}\lambda\,\diff\nu(\lambda)\) is finite, the memory potential and first-moment field possess the regularity needed to identify the semigroup solution with an encoded weak formulation and to obtain explicit stability bounds.  We further prove uniform norm-resolvent convergence and convergence of the associated semigroups when a coercive instantaneous contribution vanishes.  Under an additional $L^{2}(0,\Tend;V)$-regularity assumption on the limiting solution, the convergence rate in the memory graph norm is $O(\varepsilon^{1/2})$.  These results provide a continuous stability target for structure-preserving and certified discretisations of memory-dominated diffusion.
\end{abstract}




\section{Introduction}
Let $\Omega\subset\R^{d}$, $d\in\{1,2,3\}$, be a bounded Lipschitz domain, let $H=L^{2}(\Omega)$ and $V=H^{1}_{0}(\Omega)$, and consider
\begin{align}\label{eq:model}
\partial_{t}u(t)+\mathsf A_{0}u(t)+\mathsf A_{1}(k*u)(t)
&=f(t) \quad\text{in }V',
&u(0)&=u_{0},
\end{align}
on a finite interval $(0,\Tend)$.  The operators $\mathsf A_{0},\mathsf A_{1}:V\to V'$ are induced by bounded symmetric bilinear forms $a_{0}$ and $a_{1}$, while $k$ is a locally integrable completely monotone kernel.  Such equations arise in heat conduction with memory \cite{GurtinPipkin1968}, in viscoelasticity and materials with fading memory \cite{Dafermos1970,AmendolaFabrizioGolden2012}, and in Maxwell-type models in which the absence of solvent viscosity removes the instantaneous coercive contribution \cite{RenardyHrusaNohel1987,WangRenardy2011}.

When $a_{0}$ is coercive, the standard energy argument controls $u$ in $L^{2}(0,\Tend;V)$ and the memory term may be treated as a positive-type perturbation.  The situation changes when $a_{0}$ is only non-negative.  A positive-type memory term remains dissipative, but it does not supply a frequency-uniform lower bound in the instantaneous energy norm.  More precisely, the companion structural paper \cite{Ishizaka2026Coercivity} proves that, for every non-trivial locally integrable completely monotone kernel and every non-zero memory form, the quotient of memory dissipation by the instantaneous $a_{1}$-energy has infimum zero on every fixed time interval.  We recall this no-go result in \cref{thm:structural-nogo}.  Thus, the lost $L^{2}(0,\Tend;V)$-coercivity cannot be recovered from memory alone.

The purpose of the present paper is to identify the state space in which the degenerate equation is nevertheless well posed.  Writing the kernel as
\[
k(t)=\int_{[0,\infty)}e^{-\lambda t}\,\diff\nu(\lambda),
\]
we adjoin the internal-variable family
\[
\zeta(\lambda,t)=\int_{0}^{t}e^{-\lambda(t-s)}u(s)\,\diff s.
\]
If the representing measure has finite total mass $M_{0}=\nu([0,\infty))$, the internal variables belong to a Hilbert space $\mathcal K$ and the augmented state $(u,\zeta)$ belongs to $\Wsp=H\times\mathcal K$.  The aggregation operator $J\zeta=\int\zeta\,\diff\nu$ and the constant embedding $Eu(\lambda)=u$ satisfy an exact adjoint relation.  Consequently, the two coupling terms cancel in the energy identity and the augmented generator is $m$-dissipative.  This gives a contraction semigroup and Hadamard well-posedness without using any positive coercivity constant for $a_{0}$; see \cref{thm:graphwp,cor:hadamard}.

The extended-state construction also clarifies what is meant here by a \emph{memory graph space}.  The ambient space $\Wsp$ is the state space for the semigroup.  The zero-prehistory solution trajectories form the graph of the causal internal-variable map $u\mapsto\zeta_{u}$ inside $C([0,\Tend];\Wsp)$.  The graph norm is precisely the supremum in time of the extended memory energy.  This terminology is therefore distinct from the operator graph norm of the generator.

A second objective is to connect the semigroup solution with a weak formulation that remains meaningful even when $u\notin L^{2}(0,\Tend;V)$.  Under the additional first-moment condition
\[
M_{1}=\int_{[0,\infty)}\lambda\,\diff\nu(\lambda)<\infty,
\]
the memory potential $\xi=\int\zeta\,\diff\nu$ and the first-moment field $\eta=\int\lambda\zeta\,\diff\nu$ satisfy
\[
\partial_{t}\xi=M_{0}u-\eta.
\]
This identity allows the $a_{0}$-term to be transferred from $u$ to the $V$-valued fields $\xi$ and $\eta$.  We obtain an encoded weak formulation, identify the graph-space mild solution with it, and derive explicit stability bounds independent of the instantaneous coercivity; see \cref{def:weak,thm:exist,cor:cert}.

Finally, we treat the actual limit in which instantaneous coercivity vanishes.  For $a_{0}^{\varepsilon}=a_{0}+\varepsilon a_{v}$, with $a_{v}$ coercive, the corresponding augmented generators act on the same memory state space.  We prove norm-resolvent convergence at rate $O(\varepsilon)$ and deduce convergence of the semigroups and inhomogeneous solutions in $C([0,\Tend];\Wsp)$ from an explicit resolvent commutator identity rather than from an appeal to the Trotter--Kato theorem; on resolvent-smoothed initial states the semigroup convergence then inherits the rate $O(\varepsilon)$.  If the limiting physical component has the additional regularity $u^{0}\in L^{2}(0,\Tend;V)$, an energy argument gives an $O(\varepsilon^{1/2})$ convergence rate in the memory graph norm.

The moment restrictions separate the scope of the results.  Finite $M_{0}$ is sufficient for the extended-state semigroup theory.  Finite $M_{1}$ is used only for the encoded weak formulation and the explicit estimates involving $\eta$.  Weakly singular fractional kernels have $M_{0}=M_{1}=\infty$ and lie outside this finite-mass state space, although their solvability can be treated within classical abstract Volterra frameworks \cite{GripenbergLondenStaffans1990,Pruss1993}; see also \cite{Zacher2009}.  Constructing an equally explicit coercivity-robust state space for such infinite-mass kernels remains open in the present framework.

\Cref{sec:prelim} fixes the notation, records the positive-type identity, and recalls the structural no-go theorem.  \Cref{sec:existence} constructs the internal-variable state space and proves semigroup well-posedness.  \Cref{sec:weak} develops the encoded weak formulation and its identification with the mild solution.  \Cref{sec:certified} gives explicit graph-space stability, \cref{sec:vanishing} proves the vanishing-coercivity limit, and \cref{sec:discretisation} records the consequences for certified discretisation.

\section{Notation and preliminaries}\label{sec:prelim}
Let $A_{0},A_{1}\in L^{\infty}(\Omega)^{d\times d}$ be symmetric matrix fields and define
\begin{align*}
\displaystyle
  a_{i}(w,v):=\int_{\Omega}A_{i}(x)\nabla w\cdot\nabla v\diff x, \quad i\in\{0,1\}.
\end{align*}
We recall $\mathsf{A}_{i}\colon V\to V'$ for the spatial operator induced by the form, $\dual{\mathsf{A}_{i}w}{v}=a_{i}(w,v)$ ($i=0,1$); thus $\mathsf{A}_{i}$ is generated by the coefficient field $A_{i}$, which we keep distinct in notation. When $d=1$, the coefficient fields $A_{i}$ are scalar functions, $L^{2}(\Omega)^{d}$ is identified with $L^{2}(\Omega)$, and $\nabla u$ is simply $u'$; all matrix products below then reduce to ordinary multiplication. We assume throughout that the forms $a_{0},a_{1}$ are bounded and that the coefficient field $A_{1}$ is symmetric and pointwise positive semidefinite for almost every $x\in\Omega$, written $A_{1}\succeq0$; explicitly, $A_{1}(x)=A_{1}(x)^{\top}$ and $\xi^{\top}A_{1}(x)\,\xi \geq 0$ for every $\xi\in\R^{d}$ and almost every $x\in\Omega$. Because the matrix square root is a continuous function on symmetric positive-semidefinite matrices, this defines a symmetric, bounded, measurable field $A_{1}^{1/2}\in L^{\infty}(\Omega)^{d\times d}$ with $\left(A_{1}^{1/2}(x)\right)^{2}=A_{1}(x)$ for almost every $x$; in particular
\begin{align*}
\displaystyle
a_{1}(v,v)=\int_{\Omega}A_{1}\nabla v\cdot\nabla v\diff x =\int_{\Omega}\left\lvert A_{1}^{1/2}\nabla v\right\rvert^{2}\diff x \geq 0 \quad \forall v\in V,
\end{align*}
so the non-negativity of $a_{1}$ is a consequence of $A_{1}\succeq0$ rather than a separate hypothesis. We do \emph{not} assume coercivity of $a_{0}$ unless it is stated explicitly. The associated energy field of a function $u\colon(0,\Tend]\to V$ is
\begin{align}\label{eq:Wfield}
\displaystyle
W(t):=A_{1}^{1/2}\nabla u(t)\in\Lp, \quad a_{1}\left(u(s),u(t)\right)=\left(W(s),W(t)\right),
\end{align}
where $(\cdot,\cdot)$ and $\norm{\cdot}$ denote, here and below, the inner product and norm of $\Lp$; the second identity in \cref{eq:Wfield} is the polarisation of $a_{1}(v,v)=\norm{A_{1}^{1/2}\nabla v}^{2}$ and uses the symmetry of $A_{1}^{1/2}$. In particular $a_{1}(u(t),u(t))=\norm{W(t)}^{2}$.

\subsection{Completely monotone kernels}
We restrict attention to \emph{locally integrable} completely monotone kernels: completely monotone densities $k$ with $k\in L^{1}(0,\Tend)$ for every $\Tend>0$, equivalently with representing measure $\nu$ satisfying $\int_{[1,\infty)}\lambda^{-1}\diff\nu(\lambda)<\infty$. The associated measure is $\diff\mu(s)=k(s)\diff s$ on $(0,\Tend]$. This class contains the fractional kernels $t^{-\alpha}/\Gamma(1-\alpha)$, $\alpha\in(0,1)$, while excluding non-integrable completely monotone densities such as $k(t)=1/t$. Here, $\Gamma$ denotes the Gamma function, $\Gamma(z) := \int_0^{\infty} t^{z-1} e^{-t} \diff t$ for $z > 0$. We assume throughout that the kernel is non-trivial, $k\not\equiv0$; equivalently, its representing measure satisfies $\nu\neq0$.

\begin{definition}[Completely monotone kernel]\label{def:cm}
A function $k\colon(0,\infty)\to[0,\infty)$ is \emph{completely monotone} if it is of class $C^{\infty}$ and $(-1)^{n}k^{(n)}(t) \geq 0$ for all $t>0$ and $n\in\N_{0}$. By Bernstein's theorem \cite[Thm.~1.4]{SchillingSongVondracek2012}, this holds if and only if there exists a non-negative Borel measure $\nu$ on $[0,\infty)$ such that
\begin{align}\label{eq:bernstein}
\displaystyle
k(t)=\int_{[0,\infty)} e^{-\lambda t}\diff\nu(\lambda), \quad t>0.
\end{align}
We call $\nu$ the representing measure of $k$ in Bernstein's theorem. The total mass of the kernel is $\norm{k}_{L^{1}(0,\infty)}=\int_{0}^{\infty}k(t)\diff t =\int_{[0,\infty)}\lambda^{-1}\diff\nu(\lambda)\in(0,\infty]$, by Tonelli's theorem.
\end{definition}

\begin{example}[Standard cases]\label{ex:kernels}
\leavevmode
\begin{itemize}[leftmargin=1.6em]
  \item \emph{Exponential (single relaxation time):} $k(t)=\gamma e^{-\gamma t}$ with $\gamma>0$. Then, $\nu=\gamma\,\delta_{\gamma}$ and $\norm{k}_{L^{1}(0,\infty)}=1$.
  \item \emph{Fractional:} $k(t)=t^{-\alpha}/\Gamma(1-\alpha)$ with $\alpha\in(0,1)$. Then, $\diff\nu_{\alpha}(\lambda) =\tfrac{\sin(\pi\alpha)}{\pi}\lambda^{\alpha-1}\diff\lambda$ and
        $\norm{k}_{L^{1}(0,\infty)}=\infty$; see, e.g., \cite{GorenfloKilbasMainardiRogosin2014} for the fractional-calculus background.
\end{itemize}
\end{example}

\begin{example}[Prony kernel / finite relaxation spectrum]\label{ex:prony}
Let $k(t)=\sum_{j=1}^{J}c_{j}e^{-\gamma_{j}t}$ with $c_{j}>0$ and distinct $\gamma_{j}>0$. Then, $k$ is completely monotone with representing measure $\nu=\sum_{j=1}^{J}c_{j}\delta_{\gamma_{j}}$, and
\begin{align*}
M_{0}
&=
\sum_{j=1}^{J}c_{j}
=
k(0^{+}),
&
M_{1}
&=
\sum_{j=1}^{J}c_{j}\gamma_{j}
=
-k'(0^{+}),
\\
\norm{k}_{L^{1}(0,\infty)}
&=
\sum_{j=1}^{J}\frac{c_{j}}{\gamma_{j}}.
\end{align*}
Both Bernstein moments $M_{0}$ and $M_{1}$ are finite.  Hence the semigroup result of \cref{thm:graphwp}, the weak-solution identification of \cref{thm:exist}, and the explicit stability bounds of \cref{cor:cert} all apply.  Notice that the zeroth Bernstein moment $M_{0}=\sum_{j}c_{j}$ and the time integral $\norm{k}_{L^{1}(0,\infty)}=\sum_{j}c_{j}/\gamma_{j}$ are different quantities; the present graph-space theory is governed by the former.
\end{example}

Because $k$ is locally integrable, $\mu$ is finite on $(0,\Tend]$: $\mu((0,\Tend])=\int_{0}^{\Tend}k(s)\diff s<\infty$. We work with zero prehistory, so that $u$ is extended by zero to negative times, and the memory operator is
\begin{align}\label{eq:Kmu}
\displaystyle
\dual{(\Km u)(t)}{v}:=\int_{0}^{t}k(t-s)\,a_{1}\left(u(s),v\right)\diff s,
  \quad v\in V,
\end{align}
and the cumulative memory dissipation is
\begin{subequations} \label{eq:Dmu}
\begin{align}
\displaystyle
\Dm[u](\Tend) &:=\int_{0}^{\Tend}\dual{(\Km u)(t)}{u(t)}\diff t
  =\int_{0}^{\Tend}\left((k*W)(t),\,W(t)\right)\diff t, \label{eq:Dmu=a} \\
 (k*W)(t) &:=\int_{0}^{t}k(t-s)W(s)\diff s, \label{eq:Dmu=b}
\end{align}
\end{subequations}
where the second equality uses \cref{eq:Wfield}.

\subsection{Positive-type dissipation and the structural obstruction}
For a function $u$ with energy field $W=A_{1}^{1/2}\nabla u\in L^{2}(0,\Tend;\Lp)$, define
\begin{align*}
Z_{W}(\lambda,t)
:=
\int_{0}^{t}e^{-\lambda(t-s)}W(s)\,\diff s.
\end{align*}
Then $\partial_{t}Z_{W}+\lambda Z_{W}=W$ and $Z_{W}(\lambda,0)=0$.  The Bernstein representation and the one-mode energy identity give the following formula.

\begin{lemma}[Positive-type identity]\label{lem:positive-type}
For every locally integrable completely monotone kernel and every admissible $u$,
\begin{align}\label{eq:positive-type}
\Dm[u](\Tend)
=
\int_{[0,\infty)}
\left[
\frac12\norm{Z_{W}(\lambda,\Tend)}_{\Lp}^{2}
+
\lambda\int_{0}^{\Tend}\norm{Z_{W}(\lambda,t)}_{\Lp}^{2}\,\diff t
\right]\diff\nu(\lambda)
\geq0.
\end{align}
\end{lemma}

\begin{proof}
For the restriction $\nu_{R}=\nu|_{[0,R]}$, Fubini's theorem and
$W=\partial_{t}Z_{W}+\lambda Z_{W}$ yield
\begin{align*}
\int_{0}^{\Tend}((k_{R}*W)(t),W(t))\,\diff t
=
\int_{[0,R]}
\left[
\frac12\norm{Z_{W}(\lambda,\Tend)}_{\Lp}^{2}
+
\lambda\int_{0}^{\Tend}\norm{Z_{W}(\lambda,t)}_{\Lp}^{2}\,\diff t
\right]\diff\nu(\lambda).
\end{align*}
Since $k_{R}\uparrow k$ and $k_{R}\to k$ in $L^{1}(0,\Tend)$, Young's inequality passes the left-hand side to the limit, while monotone convergence applies on the right.
\end{proof}

The identity proves positivity, but it does not give a positive lower bound in the instantaneous energy norm.  The following result is the structural input that motivates the graph-space formulation.

\begin{theorem}[No instantaneous coercivity from memory]\label{thm:structural-nogo}
Let $\Tend>0$, let $k$ be a non-trivial locally integrable completely monotone kernel, and assume that $a_{1}\not\equiv0$.  Then
\begin{align}\label{eq:structural-nogo}
\inf_{u}
\frac{\Dm[u](\Tend)}
{\displaystyle\int_{0}^{\Tend}a_{1}(u(t),u(t))\,\diff t}
=0,
\end{align}
where the infimum is over all $u$ for which the denominator is positive.  Consequently, no $c>0$ can satisfy
\[
\Dm[u](\Tend)
\geq
c\int_{0}^{\Tend}a_{1}(u(t),u(t))\,\diff t
\]
for every admissible state.  If $a_{1}$ is coercive on $V$, there is likewise no $c>0$ such that
\[
\Dm[u](\Tend)
\geq
c\norm{u}_{L^{2}(0,\Tend;V)}^{2}
\]
for every admissible state.
\end{theorem}

\begin{proof}
Choose $w_{0}\in V$ with $c_{0}:=a_{1}(w_{0},w_{0})>0$ and set
$u_{\omega}(t)=\cos(\omega t)w_{0}$.  For one relaxation mode, the scalar internal variable is
\[
z_{\omega}(\lambda,t)
=
\frac{\lambda\cos(\omega t)+\omega\sin(\omega t)-\lambda e^{-\lambda t}}
{\lambda^{2}+\omega^{2}}.
\]
A direct estimate gives
\begin{align*}
\frac12z_{\omega}(\lambda,\Tend)^{2}
+
\lambda\int_{0}^{\Tend}z_{\omega}(\lambda,t)^{2}\,\diff t
\leq
\frac{3+2\Tend\lambda}{\lambda^{2}+\omega^{2}}.
\end{align*}
Define
\[
m(\omega)
:=
\int_{[0,\infty)}\frac{\lambda}{\lambda^{2}+\omega^{2}}\,\diff\nu(\lambda).
\]
Local integrability of $k$ implies $\nu([0,1))<\infty$ and
$\int_{[1,\infty)}\lambda^{-1}\,\diff\nu(\lambda)<\infty$; a compact--tail decomposition therefore gives $m(\omega)\to0$ as $\omega\to\infty$.  Using \cref{eq:positive-type},
\begin{align*}
0
\leq
\Dm[u_{\omega}](\Tend)
\leq
c_{0}\left[
\frac{3\nu([0,1))}{\omega^{2}}
+(3+2\Tend)m(\omega)
\right]
\longrightarrow0.
\end{align*}
On the other hand,
\[
\int_{0}^{\Tend}a_{1}(u_{\omega},u_{\omega})\,\diff t
=
c_{0}\left(\frac{\Tend}{2}+\frac{\sin(2\omega\Tend)}{4\omega}\right)
\longrightarrow
\frac{c_{0}\Tend}{2}>0.
\]
The quotient therefore tends to zero.  The final assertion follows from the coercivity of $a_{1}$ and the same sequence.
\end{proof}

\begin{remark}[Relation with the companion coercivity paper]\label{rem:companion-structure}
\Cref{thm:structural-nogo} is the only part of the detailed coercivity theory needed here.  Exact gap identities, fixed-horizon thresholds, the coercivity-gap index, and singular limits of the kernel are developed separately in \cite{Ishizaka2026Coercivity}.  The present paper starts from the obstruction \cref{eq:structural-nogo} and constructs the state space in which well-posedness survives.
\end{remark}

\section{Graph-space semigroup formulation}\label{sec:existence}
\Cref{thm:structural-nogo} rules out a particular coercivity estimate; it does not imply that the degenerate problem is ill-posed. More precisely, the memory dissipation does not control the instantaneous norm $L^{2}(0,\Tend;V)$. Thus, when the coercivity of $a_0$ is lost, the standard energy space is no longer the natural space for the problem.

The appropriate replacement is obtained from the internal-variable representation. After adjoining the internal variables to the physical state $u$, the memory equation becomes an augmented first-order evolution equation. The corresponding generator is $m$-dissipative on the extended memory state space $\mathcal W$, and therefore yields existence, uniqueness, and Lipschitz continuous dependence on the data. The problem is thus well posed in the sense of Hadamard without using a positive coercivity constant for $a_0$.

Two moment conditions on the representing measure play different roles. The finiteness of
\begin{align*}
\displaystyle
M_0=\nu([0,\infty))=k(0^+)
\end{align*}
is sufficient for the graph-space well-posedness result of \cref{thm:graphwp}. The stronger condition
\begin{align*}
\displaystyle
M_1=\int_{[0,\infty)}\lambda\diff\nu(\lambda)<\infty
\end{align*}
provides additional regularity for the first-moment field and leads to the weak formulation and stability estimates developed below. Weakly singular fractional kernels have $M_0=\infty$ and therefore lie outside the present theory; see \cref{rem:fracopen}.

The following standing assumptions will be used throughout the remainder of the paper, with the moment conditions imposed separately.
\begin{itemize}[leftmargin=1.6em]
\item[(B1)]
The forms $a_0$ and $a_1$ are bounded on $V\times V$. Furthermore,
\begin{align*}
\displaystyle
a_0(v,v)\geq0
\quad
\text{for any }v\in V,
\end{align*}
but no positive lower bound is assumed for $a_0$, whereas $a_1$ is coercive:
\begin{align*}
\displaystyle
a_1(v,v)\geq
\beta\seminorm{v}_{V}^{2},
\quad
\beta>0.
\end{align*}

\item[(B2)]
The kernel $k$ is completely monotone and non-zero, with representing measure $\nu$. We set
\begin{align*}
\displaystyle
M_0
:=
\nu([0,\infty))
=
k(0^+)
\in(0,\infty],
\end{align*}
and
\begin{align*}
\displaystyle
M_1
:=
\int_{[0,\infty)}
\lambda\diff\nu(\lambda)
=
-k'(0^+)
\in[0,\infty].
\end{align*}
Furthermore, $M_1>0$ if and only if $\nu((0,\infty))>0$; the purely constant kernel, $\nu=c\delta_0$ with $c>0$, has $M_1=0$ and is not excluded from the present section.

\item[(B3)]
The data satisfy
\begin{align*}
\displaystyle
f\in L^2(0,\Tend;H),
\quad
u_0\in H.
\end{align*}
\end{itemize}

The quantity $M_0$ is the total mass of the representing measure and should not be confused with the time integral of the kernel. Whenever $k\in L^1(0,\infty)$,
\begin{align*}
\displaystyle
\norm{k}_{L^1(0,\infty)}
=
\int_{(0,\infty)}
\frac{1}{\lambda}\diff\nu(\lambda)
=
m(0).
\end{align*}

The moment assumptions will be imposed separately. The semigroup theorem \cref{thm:graphwp} requires $M_0<\infty$. The weak-solution identification of \cref{thm:exist}, together with the additional $V$-regularity of the first-moment field, requires both
\begin{align*}
\displaystyle
M_0<\infty
\quad\text{and}\quad
M_1<\infty.
\end{align*}

The assumption $f\in L^2(0,\Tend;H)$ is stronger than the usual parabolic assumption $f\in L^2(0,\Tend;V')$. It is required because the degenerate theory does not initially provide $u\in L^2(0,\Tend;V)$, so the $V'$-$V$ duality pairing is not available in the basic energy estimate. The $H$-valued assumption allows the forcing term to be paired instead with $u\in L^\infty(0,\Tend;H)$.

\subsection{The internal-variable family and the memory potential}
The memory term is non-local in time. We localise it by decomposing the kernel into a continuum of exponentially relaxing modes. For the moment, let $u$ be sufficiently regular for the following Bochner integrals and changes in the order of integration to be justified.

For each relaxation rate $\lambda\geq0$, define the internal variable
\begin{align}\label{eq:zetafam}
\zeta(\lambda,t)
:=
\int_0^t
e^{-\lambda(t-s)}u(s)\diff s.
\end{align}
For fixed $\lambda$, this is the unique solution of
\begin{align}\label{eq:zeta-evolution}
\partial_t\zeta(\lambda,t)
+
\lambda\zeta(\lambda,t)
&=
u(t),
\\
\zeta(\lambda,0)
&=
0.
\end{align}
Thus, $\zeta(\lambda,\cdot)$ represents the response of a single relaxation mode with decay rate $\lambda$.

Using the integral representation
\begin{align*}
\displaystyle
k(r)
=
\int_{[0,\infty)}
e^{-\lambda r}\diff\nu(\lambda),
\end{align*}
together with Fubini's theorem, we obtain
\begin{align*}
\int_0^t
k(t-s)u(s)\diff s
&=
\int_0^t
\int_{[0,\infty)}
e^{-\lambda(t-s)}u(s)
\diff\nu(\lambda)\diff s
=
\int_{[0,\infty)}
\zeta(\lambda,t)\diff\nu(\lambda).
\end{align*}
This motivates the definition of the \emph{memory potential}
\begin{align}\label{eq:xidef}
\xi(t)
:=
\int_{[0,\infty)}
\zeta(\lambda,t)\diff\nu(\lambda)
=
(k*u)(t).
\end{align}
In particular, for any $v\in V$,
\begin{align}\label{eq:history-xi}
\int_0^t
k(t-s)a_1(u(s),v)\diff s
=
a_1(\xi(t),v).
\end{align}
Therefore, $\xi$ is precisely the aggregated memory field that enters the equation. Because any internal variable initially vanishes, one also has
\begin{align*}
\displaystyle
\xi(0)=0.
\end{align*}

We shall also use the \emph{first-moment field}
\begin{align}\label{eq:etadef}
\eta(t)
:=
\int_{[0,\infty)}
\lambda\zeta(\lambda,t)\diff\nu(\lambda),
\end{align}
whenever this integral is well defined. The fields $\xi$ and $\eta$ play different roles. $\xi$ appears directly in the memory term, whereas $\eta$ appears when the memory potential is differentiated. Indeed, assume that
\begin{align*}
\displaystyle
M_0
=
\nu([0,\infty))
<
\infty
\end{align*}
and that differentiation under the $\nu$-integral is justified. Using
\cref{eq:zeta-evolution}, we obtain
\begin{align}\label{eq:xidot}
\partial_t\xi(t)
&=
\int_{[0,\infty)}
\partial_t\zeta(\lambda,t)\diff\nu(\lambda)
=
\int_{[0,\infty)}
\left (u(t)-\lambda\zeta(\lambda,t)\right)
\diff\nu(\lambda)
=
M_0u(t)-\eta(t).
\end{align}
Because the memory is non-trivial, $M_0>0$. Consequently,
\begin{align}\label{eq:u-from-xi}
u(t)
=
\frac{1}{M_0}
\left(\partial_t\xi(t)+\eta(t)\right).
\end{align}

The original memory equation can therefore be written as the augmented system
\begin{align}\label{eq:augmented-system}
\partial_tu+\mathsf A_0u+\mathsf A_1\xi
&=
f,
\\
\partial_t\zeta(\lambda,\cdot)
+
\lambda\zeta(\lambda,\cdot)
&=
u,
\quad \lambda\geq0,
\nonumber\\
\xi
&=
\int_{[0,\infty)}
\zeta(\lambda,\cdot)\diff\nu(\lambda).
\nonumber
\end{align}
Although this system contains a continuum of internal variables, it is local in time on the enlarged state space. This local first-order formulation is the basis of the graph-space well-posedness theory below.

\subsection{Well-posedness in the extended memory space}
\label{sec:graphwp}
The no-go theorem does not exclude additional $L^{2}(0,\Tend;V)$-regularity of solutions. It shows instead that such regularity cannot be obtained from the memory dissipation through a frequency-uniform coercivity estimate. When the instantaneous form $a_0$ is only non-negative, the natural energy must therefore be formulated in terms of the internal variables.

Assume throughout this subsection that
\begin{align*}
\displaystyle
0<M_0=\nu([0,\infty))<\infty.
\end{align*}
Let $\mathcal K$ be the Hilbert space of strongly $\nu$-measurable functions $\zeta\colon[0,\infty)\to V$ such that
\begin{align*}
\displaystyle
\norm{\zeta}_{\mathcal K}
:=
\left( \int_{[0,\infty)}
a_1\left(\zeta(\lambda),\zeta(\lambda)\right)
\diff\nu(\lambda) \right)^{1/2}
<\infty.
\end{align*}
Because $a_1$ is bounded and coercive on $V$, this norm is equivalent to
\begin{align*}
\displaystyle
\left(
\int_{[0,\infty)}
\norm{\zeta(\lambda)}_V^2
\diff\nu(\lambda)
\right)^{1/2}.
\end{align*}
We introduce the extended memory space
\begin{align}\label{eq:graphspace}
\Wsp
:=
H\times\mathcal K,
\end{align}
equipped with the inner product
\begin{align*}
\dual{(u,\zeta)}{(v,\chi)}_{\Wsp}
:=
(u,v)_H
+
\int_{[0,\infty)}
a_1\left(\zeta(\lambda),\chi(\lambda)\right)
\diff\nu(\lambda).
\end{align*}
We call $\Wsp$ the \emph{extended memory state space}.

Two auxiliary objects make the trajectory description of zero-prehistory solutions precise.  Let $\mathcal H_{\nu}$ be the Hilbert space of strongly $\nu$-measurable functions $z\colon[0,\infty)\to H$ with
\begin{align*}
\inner{z}{w}_{\mathcal H_{\nu}}
:=
\int_{[0,\infty)}
(z(\lambda),w(\lambda))_H
\diff\nu(\lambda),
\qquad
\norm{z}_{\mathcal H_{\nu}}
:=
\inner{z}{z}_{\mathcal H_{\nu}}^{1/2}
<\infty,
\end{align*}
and, for $u\in C([0,\Tend];H)$, define the causal internal-variable map
\begin{align}\label{eq:causal-map}
(\mathscr Ru)(t)(\lambda)
:=
\int_0^t
e^{-\lambda(t-s)}u(s)\diff s,
\qquad
\lambda\geq0,
\quad
t\in[0,\Tend].
\end{align}

\begin{lemma}[Internal-variable image space and causal map]\label{lem:causal-map}
Assume \textnormal{(B1)} and $0<M_0=\nu([0,\infty))<\infty$, and let $C_P$ denote the Poincar\'e constant of $V=H_0^1(\Omega)$ in $H$, so that $\norm{v}_H\leq C_P\seminorm{v}_V$ for every $v\in V$.
\begin{enumerate}[label=\textnormal{(\roman*)}]
\item The space $\mathcal K$ is continuously embedded into $\mathcal H_{\nu}$:
\begin{align}\label{eq:K-Hnu-embedding}
\norm{z}_{\mathcal H_{\nu}}^{2}
\leq
\frac{C_P^{2}}{\beta}
\norm{z}_{\mathcal K}^{2}
\qquad
\text{for every }z\in\mathcal K.
\end{align}
\item For every $u\in C([0,\Tend];H)$, the trajectory $\mathscr Ru$ belongs to $C([0,\Tend];\mathcal H_{\nu})$, and
\begin{align}\label{eq:causal-bound}
\sup_{0\leq t\leq\Tend}
\norm{(\mathscr Ru)(t)}_{\mathcal H_{\nu}}
\leq
M_0^{1/2}\,\Tend
\norm{u}_{C([0,\Tend];H)}.
\end{align}
In particular, $\mathscr R\colon C([0,\Tend];H)\to C([0,\Tend];\mathcal H_{\nu})$ is linear and bounded.
\end{enumerate}
\end{lemma}

\begin{proof}
(i) For $\nu$-almost every $\lambda$, the Poincar\'e inequality and the coercivity of $a_1$ give
\begin{align*}
\norm{z(\lambda)}_H^{2}
\leq
C_P^{2}\seminorm{z(\lambda)}_V^{2}
\leq
\frac{C_P^{2}}{\beta}
a_1(z(\lambda),z(\lambda)).
\end{align*}
Integration with respect to $\nu$ proves \eqref{eq:K-Hnu-embedding}.

(ii) For every $\lambda\geq0$ and every $t\in[0,\Tend]$,
\begin{align}\label{eq:causal-mode-bound}
\norm{(\mathscr Ru)(t)(\lambda)}_H
\leq
\int_0^te^{-\lambda(t-s)}\diff s\,
\norm{u}_{C([0,\Tend];H)}
\leq
\min\{t,\lambda^{-1}\}
\norm{u}_{C([0,\Tend];H)}.
\end{align}
Because $\nu$ is finite, squaring and integrating with respect to $\nu$ gives \eqref{eq:causal-bound}.  For the continuity in time, let $0\leq t\leq t+h\leq\Tend$.  Splitting the time integral at $s=t$ yields
\begin{align*}
(\mathscr Ru)(t+h)(\lambda)
-
(\mathscr Ru)(t)(\lambda)
=
\bigl(e^{-\lambda h}-1\bigr)
(\mathscr Ru)(t)(\lambda)
+
\int_t^{t+h}
e^{-\lambda(t+h-s)}u(s)\diff s.
\end{align*}
The last term has $H$-norm at most $h\norm{u}_{C([0,\Tend];H)}$ uniformly in $\lambda$, so its $\mathcal H_{\nu}$-norm is at most $M_0^{1/2}h\norm{u}_{C([0,\Tend];H)}$.  By \eqref{eq:causal-mode-bound}, the first term is bounded in $H$ by $\Tend\norm{u}_{C([0,\Tend];H)}$, which is square-integrable with respect to the finite measure $\nu$, and it tends to zero in $H$ for every $\lambda$ as $h\downarrow0$; the dominated convergence theorem therefore gives convergence to zero in $\mathcal H_{\nu}$.  Reading the same decomposition from the smaller of two times gives left-continuity.  Linearity is clear, and boundedness is \eqref{eq:causal-bound}.
\end{proof}

For the fixed time horizon $\Tend$, the zero-prehistory \emph{memory graph space} is
\begin{align}\label{eq:trajectory-graph}
\mathfrak G_{\nu}(0,\Tend)
:=
\Bigl\{
(u,\zeta)\in C([0,\Tend];\Wsp):
\zeta(t)
=
(\mathscr Ru)(t)
\ \text{in }\mathcal H_{\nu}
\text{ for every $t\in[0,\Tend]$}
\Bigr\},
\end{align}
with norm
\begin{align}\label{eq:trajectory-graph-norm}
\norm{(u,\zeta)}_{\mathfrak G_{\nu}(0,\Tend)}
:=
\sup_{0\leq t\leq\Tend}
\norm{(u(t),\zeta(t))}_{\Wsp}.
\end{align}
Membership in $\mathfrak G_{\nu}(0,\Tend)$ states that, for every $t\in[0,\Tend]$,
\begin{align*}
\zeta(\lambda,t)
=
\int_0^te^{-\lambda(t-s)}u(s)\diff s
\quad
\text{in }H
\quad
\text{for $\nu$-a.e. }\lambda.
\end{align*}
Because $\mathcal K$ is continuously embedded into $\mathcal H_{\nu}$ and $\mathscr R$ is bounded, by \cref{lem:causal-map}, the constraint $\zeta(t)=(\mathscr Ru)(t)$ is preserved under convergence in $C([0,\Tend];\Wsp)$; hence $\mathfrak G_{\nu}(0,\Tend)$ is a closed subspace of $C([0,\Tend];\Wsp)$ and, in particular, a Banach space with the norm \cref{eq:trajectory-graph-norm}.  Thus, $\mathfrak G_{\nu}(0,\Tend)$ is the graph of the causal internal-variable map inside the trajectory space $C([0,\Tend];\Wsp)$.  The term \emph{memory graph norm} refers to \cref{eq:trajectory-graph-norm}; it does not refer to the operator graph norm $\norm{X}_{\Wsp}+\norm{\mathcal AX}_{\Wsp}$ of the generator introduced below.
We define the aggregation operator as
\begin{align*}
\displaystyle
J\colon\mathcal K\to V,
\quad
J\zeta
:=
\int_{[0,\infty)}
\zeta(\lambda)\diff\nu(\lambda),
\end{align*}
and the constant embedding
\begin{align*}
\displaystyle
E\colon V\to\mathcal K,
\quad
(Eu)(\lambda):=u.
\end{align*}
Both operators are bounded. Indeed,
\begin{align}\label{eq:Jbound}
a_1(J\zeta,J\zeta)^{1/2}
&\leq
\int_{[0,\infty)}
a_1(\zeta(\lambda),\zeta(\lambda))^{1/2}
\diff\nu(\lambda)
\nonumber\\
&\leq
M_0^{1/2}\norm{\zeta}_{\mathcal K},
\end{align}
and
\begin{align}\label{eq:Ebound}
\norm{Eu}_{\mathcal K}^{2}
=
M_0a_1(u,u).
\end{align}
Furthermore,
\begin{align}\label{eq:adjoint-coupling}
a_1(J\zeta,u)
=
\int_{[0,\infty)}
a_1(\zeta(\lambda),u)\diff\nu(\lambda)
=
\dual{\zeta}{Eu}_{\mathcal K}.
\end{align}
Identity \eqref{eq:adjoint-coupling} is the source of the cancellation in the energy estimate.

Let
\begin{align*}
\displaystyle
(\Lambda\zeta)(\lambda)
:=
\lambda\zeta(\lambda).
\end{align*}
We define the augmented operator as
\begin{align}\label{eq:genA}
\mathcal A(u,\zeta)
:=
\left(
-\mathsf A_0u-\mathsf A_1J\zeta,\,
Eu-\Lambda\zeta
\right)
\end{align}
on the domain
\begin{align}\label{eq:genAdomain}
D(\mathcal A)
:=
\left\{
(u,\zeta)\in V\times\mathcal K:
\mathsf A_0u+\mathsf A_1J\zeta\in H,\ 
Eu-\Lambda\zeta\in\mathcal K
\right\}.
\end{align}
Here, the condition $\mathsf A_0u+\mathsf A_1J\zeta\in H$ is understood after identifying $H$ with a subspace of $V'$. The augmented memory equation is then the abstract Cauchy problem
\begin{align}\label{eq:abstract-cauchy}
X'(t)
=
\mathcal AX(t)+(f(t),0),
\quad
X(0)=(u_0,0),
\end{align}
where
\begin{align*}
\displaystyle
X(t)=(u(t),\zeta(t)).
\end{align*}

\begin{theorem}[Well-posedness in the extended memory space]
\label{thm:graphwp}
Assume \textnormal{(B1)} and
\begin{align*}
\displaystyle
0<M_0=\nu([0,\infty))<\infty.
\end{align*}
Then, the operator $\mathcal A$ defined by \cref{eq:genA,eq:genAdomain} is densely defined and $m$-dissipative on $\Wsp$. More precisely,
\begin{align}\label{eq:dissip}
\dual{\mathcal A(u,\zeta)}{(u,\zeta)}_{\Wsp}
=
-a_0(u,u)
-
\int_{[0,\infty)}
\lambda
a_1(\zeta(\lambda),\zeta(\lambda))
\diff\nu(\lambda)
\leq0
\end{align}
for any $(u,\zeta)\in D(\mathcal A)$, and
\begin{align*}
\displaystyle
\operatorname{Ran}(I-\mathcal A)=\Wsp.
\end{align*}
Consequently, $\mathcal A$ generates a $C_0$-semigroup of contractions $(S(t))_{t\geq0}$ on $\Wsp$. For any
\begin{align*}
\displaystyle
u_0\in H,
\quad
f\in L^1(0,\Tend;H),
\end{align*}
the problem \eqref{eq:abstract-cauchy} has a unique mild solution
\begin{align*}
\displaystyle
X=(u,\zeta)\in \mathcal{C}([0,\Tend];\Wsp),
\end{align*}
given by
\begin{align}\label{eq:variation}
X(t)
=
S(t)(u_0,0)
+
\int_0^t
S(t-s)(f(s),0)\diff s.
\end{align}
It satisfies
\begin{align}\label{eq:contraction}
\norm{X(t)}_{\Wsp}
\leq
\norm{u_0}_H
+
\int_0^t\norm{f(s)}_H\diff s.
\end{align}
In particular, the solution depends Lipschitz-continuously on the data. The estimate contains no positive lower bound for $a_0$.
\end{theorem}

\begin{proof}
We divide the proof into four steps.

\medskip
\noindent
\emph{Step 1: dissipativity.}
Let
\begin{align*}
\displaystyle
X=(u,\zeta)\in D(\mathcal A).
\end{align*}
By the definition of the inner product on $\Wsp$,
\begin{align*}
\dual{\mathcal AX}{X}_{\Wsp}
&=
\left(
-\mathsf A_0u-\mathsf A_1J\zeta,u
\right)_H
+
\dual{Eu-\Lambda\zeta}{\zeta}_{\mathcal K}.
\end{align*}
The first term is
\begin{align*}
\displaystyle
\left(
-\mathsf A_0u-\mathsf A_1J\zeta,u
\right)_H
=
-a_0(u,u)-a_1(J\zeta,u).
\end{align*}
For the second term, we have
\begin{align*}
\dual{Eu-\Lambda\zeta}{\zeta}_{\mathcal K}
&=
\int_{[0,\infty)}
a_1(u,\zeta(\lambda))
\diff\nu(\lambda)
\\
&\quad
-
\int_{[0,\infty)}
\lambda
a_1(\zeta(\lambda),\zeta(\lambda))
\diff\nu(\lambda).
\end{align*}
From symmetry of $a_1$ and \eqref{eq:adjoint-coupling},
\begin{align*}
\displaystyle
a_1(J\zeta,u)
=
\int_{[0,\infty)}
a_1(u,\zeta(\lambda))
\diff\nu(\lambda).
\end{align*}
The two coupling terms therefore cancel exactly. Therefore
\begin{align}\label{eq:dissip-identity}
\dual{\mathcal AX}{X}_{\Wsp}
=
-a_0(u,u)
-
\int_{[0,\infty)}
\lambda
a_1(\zeta(\lambda),\zeta(\lambda))
\diff\nu(\lambda).
\end{align}
Because $a_0(u,u)\geq0$ and $\lambda\geq0$, this proves \eqref{eq:dissip}.

\medskip
\noindent
\emph{Step 2: surjectivity of $I-\mathcal A$.}
Let
\begin{align*}
\displaystyle
(F,G)\in\Wsp
\end{align*}
be arbitrary. We seek $(u,\zeta)\in D(\mathcal A)$ satisfying
\begin{align*}
\displaystyle
(I-\mathcal A)(u,\zeta)=(F,G).
\end{align*}
The two components of this equation are
\begin{align}
u+\mathsf A_0u+\mathsf A_1J\zeta
&=
F,
\label{eq:resolvent-first}
\\
(1+\lambda)\zeta(\lambda)-u
&=
G(\lambda).
\label{eq:resolvent-second}
\end{align}
From \eqref{eq:resolvent-second},
\begin{align}\label{eq:zeta-resolvent}
\zeta(\lambda)
=
\frac{u+G(\lambda)}{1+\lambda}.
\end{align}
We set
\begin{align}\label{eq:cg-def}
c_\nu
:=
\int_{[0,\infty)}
\frac{1}{1+\lambda}
\diff\nu(\lambda),
\quad
h
:=
\int_{[0,\infty)}
\frac{G(\lambda)}{1+\lambda}
\diff\nu(\lambda).
\end{align}
Because the memory is non-zero and $M_0<\infty$,
\begin{align*}
\displaystyle
0<c_\nu\leq M_0.
\end{align*}
Furthermore, by \eqref{eq:Jbound},
\begin{align*}
a_1(h,h)^{1/2}
&\leq
M_0^{1/2}
\left(
\int_{[0,\infty)}
\frac{a_1(G(\lambda),G(\lambda))}
{(1+\lambda)^2}
\diff\nu(\lambda)
\right)^{1/2}
\leq
M_0^{1/2}\norm{G}_{\mathcal K}.
\end{align*}
Thus, $h\in V$. Integrating \eqref{eq:zeta-resolvent} with respect to $\nu$, we obtain
\begin{align}\label{eq:Jzeta-resolvent}
J\zeta
=
c_\nu u+h.
\end{align}
Substituting this identity into \eqref{eq:resolvent-first} gives
\begin{align*}
\displaystyle
u+\mathsf A_0u+c_\nu\mathsf A_1u
=
F-\mathsf A_1h
\end{align*}
in $V'$. Equivalently, $u\in V$ must satisfy
\begin{align}\label{eq:LM-resolvent}
(u,v)_H
+
a_0(u,v)
+
c_\nu\,a_1(u,v)
=
(F,v)_H-a_1(h,v)
\quad
\text{for any }v\in V.
\end{align}
The bilinear form
\begin{align*}
\displaystyle
b(u,v)
:=
(u,v)_H+a_0(u,v)+c_\nu\,a_1(u,v)
\end{align*}
is bounded on $V\times V$. Furthermore,
\begin{align*}
\displaystyle
b(v,v)
\geq
\norm{v}_H^2+c_\nu\beta\seminorm{v}_V^2
\geq
\min\{1,c_\nu\beta\}\norm{v}_V^2,
\end{align*}
where we use $\norm{v}_V^2=\norm{v}_H^2+\seminorm{v}_V^2$. Therefore, $b$ is coercive on $V$, and the Lax--Milgram theorem yields a unique $u\in V$ satisfying \eqref{eq:LM-resolvent}. We define
\begin{align*}
\displaystyle
\zeta(\lambda)
:=
\frac{u+G(\lambda)}{1+\lambda}.
\end{align*}
We set
\begin{align*}
\displaystyle
\zeta_u(\lambda):=\frac{u}{1+\lambda},
\quad
\zeta_G(\lambda):=\frac{G(\lambda)}{1+\lambda}.
\end{align*}
Then,
\begin{align*}
\displaystyle
\zeta=\zeta_u+\zeta_G \in\mathcal K,
\end{align*}
and therefore
\begin{align*}
\displaystyle
\norm{\zeta}_{\mathcal K}
\leq
\norm{\zeta_u}_{\mathcal K}
+
\norm{\zeta_G}_{\mathcal K}.
\end{align*}
Because $(1+\lambda)^{-2}\leq1$,
\begin{align*}
\norm{\zeta_u}_{\mathcal K}^{2}
&=
a_1(u,u)
\int_{[0,\infty)}
\frac{1}{(1+\lambda)^2}
\diff\nu(\lambda)
\leq
M_0a_1(u,u),
\end{align*}
whereas
\begin{align*}
\norm{\zeta_G}_{\mathcal K}^{2}
&=
\int_{[0,\infty)}
\frac{a_1(G(\lambda),G(\lambda))}
{(1+\lambda)^2}
\diff\nu(\lambda)
\leq
\norm{G}_{\mathcal K}^{2}.
\end{align*}
Therefore,
\begin{align*}
\displaystyle
\norm{\zeta}_{\mathcal K}
\leq
M_0^{1/2}a_1(u,u)^{1/2}
+
\norm{G}_{\mathcal K},
\end{align*}
and therefore $\zeta\in\mathcal K$. It remains to verify that $(u,\zeta)\in D(\mathcal A)$. From
\eqref{eq:resolvent-second},
\begin{align*}
\displaystyle
u-\lambda\zeta(\lambda)
=
\zeta(\lambda)-G(\lambda).
\end{align*}
Because both $\zeta$ and $G$ belong to $\mathcal K$, it follows that
\begin{align*}
\displaystyle
Eu-\Lambda\zeta\in\mathcal K.
\end{align*}
Furthermore, \eqref{eq:resolvent-first} gives
\begin{align*}
\displaystyle
\mathsf A_0u+\mathsf A_1J\zeta
=
F-u\in H.
\end{align*}
Thus, $(u,\zeta)\in D(\mathcal A)$, and
\begin{align*}
\displaystyle
\operatorname{Ran}(I-\mathcal A)=\Wsp.
\end{align*}

\medskip
\noindent
\emph{Step 3: density of the domain.}
Let $Y\in\Wsp$ satisfy
\begin{align*}
\displaystyle
\dual{Y}{X}_{\Wsp}=0
\quad
\text{for any }X\in D(\mathcal A).
\end{align*}
Because $I-\mathcal A$ is onto, there exists $X\in D(\mathcal A)$ such that
\begin{align*}
\displaystyle
(I-\mathcal A)X=Y.
\end{align*}
Because $Y$ is orthogonal to $D(\mathcal A)$, in particular, it is orthogonal to this $X$. Therefore,
\begin{align*}
0
=
\dual{Y}{X}_{\Wsp}
&=
\dual{(I-\mathcal A)X}{X}_{\Wsp}
=
\norm{X}_{\Wsp}^{2}
-
\dual{\mathcal AX}{X}_{\Wsp}.
\end{align*}
By dissipativity,
\begin{align*}
\displaystyle
-\dual{\mathcal AX}{X}_{\Wsp}\geq0,
\end{align*}
which leads to
\begin{align*}
\displaystyle
0\geq\norm{X}_{\Wsp}^{2},
\end{align*}
so $X=0$, and consequently $Y=0$. Thus,
\begin{align*}
\displaystyle
\overline{D(\mathcal A)}^{\,\Wsp}=\Wsp.
\end{align*}

\medskip
\noindent
\emph{Step 4: generation and continuous dependence.}
Steps~1--3 show that $\mathcal A$ is densely defined, dissipative, and satisfies
\begin{align*}
\displaystyle
\operatorname{Ran}(I-\mathcal A)=\Wsp.
\end{align*}
The Lumer--Phillips theorem \cite[Chap.~1, Thm.~4.3]{Pazy1983} (see also \cite{Brezis1973}) therefore implies that $\mathcal A$ generates a $C_0$-semigroup of contractions $(S(t))_{t\geq0}$ on $\Wsp$. For $u_0\in H$ and $f\in L^1(0,\Tend;H)$, the variation-of-constants formula gives
\begin{align*}
\displaystyle
X(t)
=
S(t)(u_0,0)
+
\int_0^t
S(t-s)(f(s),0)\diff s.
\end{align*}
Using
\begin{align*}
\displaystyle
\norm{S(t)}_{\mathcal L(\Wsp)}\leq1,
\end{align*}
we obtain
\begin{align*}
\norm{X(t)}_{\Wsp}
&\leq
\norm{(u_0,0)}_{\Wsp}
+
\int_0^t
\norm{(f(s),0)}_{\Wsp}\diff s
=
\norm{u_0}_H
+
\int_0^t\norm{f(s)}_H\diff s.
\end{align*}
This proves \eqref{eq:contraction}. Applying the same estimate to the difference of two solutions gives
\begin{align*}
\norm{X_1-X_2}_{C([0,\Tend];\Wsp)}
\leq
\norm{u_{0,1}-u_{0,2}}_H
+
\norm{f_1-f_2}_{L^1(0,\Tend;H)}.
\end{align*}
Thus, the solution depends Lipschitz-continuously on the initial datum and the forcing term. No positive coercivity constant for $a_0$ has been used.
\end{proof}

The mild solution of \cref{thm:graphwp} is obtained by an abstract limit procedure, so its internal component is not, a priori, given mode by mode by the causal formula \eqref{eq:zetafam}.  The next lemma closes this gap; it is the rigorous form of the variation-of-constants formula for the internal-variable equation and is used repeatedly below.

\begin{lemma}[Modewise internal-variable representation]\label{lem:modewise}
Assume \textnormal{(B1)} and $0<M_0=\nu([0,\infty))<\infty$.  For $t\geq0$, let $e^{-t\Lambda}$ denote the multiplication operator $(e^{-t\Lambda}z)(\lambda):=e^{-\lambda t}z(\lambda)$, which is a contraction on $\mathcal K$ and on $\mathcal H_{\nu}$.
\begin{enumerate}[label=\textnormal{(\roman*)}]
\item Let $X=(u,\zeta)$ be a strong solution of \eqref{eq:abstract-cauchy} on $[0,\Tend]$ with $f\in C([0,\Tend];H)$; that is, $X\in C^1([0,\Tend];\Wsp)$ satisfies $X(t)\in D(\mathcal A)$ and $X'(t)=\mathcal AX(t)+(f(t),0)$ for every $t\in[0,\Tend]$, possibly with a non-zero initial memory component $\zeta(0)$.  Then, for every $t\in[0,\Tend]$,
\begin{align}\label{eq:modewise-strong}
\zeta(t)
=
e^{-t\Lambda}\zeta(0)
+
(\mathscr Ru)(t)
\quad
\text{in }\mathcal H_{\nu}.
\end{align}
\item Let $X=(u,\zeta)$ be the mild solution \eqref{eq:variation} with data $u_0\in H$ and $f\in L^1(0,\Tend;H)$.  Then, for every $t\in[0,\Tend]$,
\begin{align}\label{eq:modewise-mild}
\zeta(t)
=
(\mathscr Ru)(t)
\quad
\text{in }\mathcal H_{\nu};
\end{align}
equivalently, for every $t\in[0,\Tend]$,
\begin{align*}
\zeta(\lambda,t)
=
\int_0^te^{-\lambda(t-s)}u(s)\diff s
\quad
\text{in }H
\quad
\text{for $\nu$-a.e. }\lambda.
\end{align*}
\end{enumerate}
\end{lemma}

\begin{proof}
\emph{Step 1: strong solutions.}
Fix $t_0\in(0,\Tend]$, $v\in H$, and a bounded Borel function $\varphi\colon[0,\infty)\to\R$ with compact support.  Define $\psi_t\in\mathcal H_{\nu}$ by
\begin{align*}
\psi_t(\lambda)
:=
\varphi(\lambda)e^{-\lambda(t_0-t)}v,
\qquad
t\in[0,t_0],
\end{align*}
and
\begin{align*}
g(t)
:=
\inner{\zeta(t)}{\psi_t}_{\mathcal H_{\nu}}
=
\int_{[0,\infty)}
\varphi(\lambda)e^{-\lambda(t_0-t)}
(\zeta(\lambda,t),v)_H
\diff\nu(\lambda).
\end{align*}
Because $\varphi$ is bounded with compact support and $\nu$ is finite, the map $t\mapsto\psi_t$ is continuously differentiable in $\mathcal H_{\nu}$ with $(\partial_t\psi_t)(\lambda)=\lambda\varphi(\lambda)e^{-\lambda(t_0-t)}v$; indeed, the difference quotients of the weight converge uniformly on the compact support of $\varphi$.  Because $X$ is a strong solution, $\zeta\in C^1([0,t_0];\mathcal K)$, and $\mathcal K$ is continuously embedded into $\mathcal H_{\nu}$ by \cref{lem:causal-map}.  The product rule for the inner product of $\mathcal H_{\nu}$ therefore gives $g\in C^1([0,t_0])$ with
\begin{align*}
g'(t)
=
\inner{\zeta'(t)}{\psi_t}_{\mathcal H_{\nu}}
+
\inner{\zeta(t)}{\partial_t\psi_t}_{\mathcal H_{\nu}}.
\end{align*}
The second component of \eqref{eq:abstract-cauchy} reads $\zeta'(t)=Eu(t)-\Lambda\zeta(t)$ in $\mathcal K$; here $u(t)\in V$ and $\Lambda\zeta(t)=Eu(t)-\zeta'(t)\in\mathcal K$ because $X(t)\in D(\mathcal A)$.  Hence, for every $t$ and $\nu$-almost every $\lambda$, $\zeta'(t)(\lambda)=u(t)-\lambda\zeta(\lambda,t)$, and
\begin{align*}
\inner{\zeta'(t)}{\psi_t}_{\mathcal H_{\nu}}
=
(u(t),v)_H
\int_{[0,\infty)}
\varphi(\lambda)e^{-\lambda(t_0-t)}
\diff\nu(\lambda)
-
\inner{\zeta(t)}{\partial_t\psi_t}_{\mathcal H_{\nu}}.
\end{align*}
The weighted terms cancel, and
\begin{align*}
g'(t)
=
(u(t),v)_H\,
\kappa_\varphi(t_0-t),
\qquad
\kappa_\varphi(r)
:=
\int_{[0,\infty)}
\varphi(\lambda)e^{-\lambda r}
\diff\nu(\lambda).
\end{align*}
Integrating over $[0,t_0]$ and interchanging the order of integration by Fubini's theorem, which is justified because the integrand is bounded on the product of $[0,t_0]$ and the compact support of $\varphi$ and $\nu$ is finite, we obtain
\begin{align*}
g(t_0)-g(0)
=
\int_0^{t_0}
(u(s),v)_H\,
\kappa_\varphi(t_0-s)
\diff s
=
\int_{[0,\infty)}
\varphi(\lambda)
\bigl(
(\mathscr Ru)(t_0)(\lambda),v
\bigr)_H
\diff\nu(\lambda).
\end{align*}
Since
\begin{align*}
g(t_0)
=
\int_{[0,\infty)}
\varphi(\lambda)
(\zeta(\lambda,t_0),v)_H
\diff\nu(\lambda),
\qquad
g(0)
=
\int_{[0,\infty)}
\varphi(\lambda)e^{-\lambda t_0}
(\zeta(\lambda,0),v)_H
\diff\nu(\lambda),
\end{align*}
the element
\begin{align*}
w
:=
\zeta(t_0)
-
e^{-t_0\Lambda}\zeta(0)
-
(\mathscr Ru)(t_0)
\in
\mathcal H_{\nu}
\end{align*}
satisfies $\int_{[0,\infty)}\varphi(\lambda)(w(\lambda),v)_H\diff\nu(\lambda)=0$ for every $v\in H$ and every bounded Borel $\varphi$ with compact support; note that $\lambda\mapsto(w(\lambda),v)_H$ is $\nu$-integrable because $w\in\mathcal H_{\nu}$ and $\nu$ is finite.  Choosing $\varphi=\mathbf 1_{B\cap[0,n]}$ for Borel sets $B\subset[0,\infty)$ and $n\in\N$ shows that $(w(\lambda),v)_H=0$ for $\nu$-almost every $\lambda$; letting $v$ run through a countable dense subset of $H$ gives $w=0$ in $\mathcal H_{\nu}$.  Since $t_0\in(0,\Tend]$ was arbitrary and \eqref{eq:modewise-strong} is trivial at $t=0$, this proves (i).

\medskip
\noindent
\emph{Step 2: mild solutions.}
Because $\mathcal A$ is $m$-dissipative and densely defined, there exist
\begin{align*}
X_{0,n}\in D(\mathcal A),
\qquad
f_n\in C^1([0,\Tend];H),
\end{align*}
such that
\begin{align*}
X_{0,n}\longrightarrow(u_0,0)
\quad\text{in }\Wsp,
\qquad
f_n\longrightarrow f
\quad\text{in }L^1(0,\Tend;H),
\end{align*}
and the corresponding strong solutions $X_n=(u_n,\zeta_n)$ converge to $X$ in $C([0,\Tend];\Wsp)$ by the contraction estimate \eqref{eq:contraction}.  By Step~1, for every $t\in[0,\Tend]$,
\begin{align*}
\zeta_n(t)
=
e^{-t\Lambda}\zeta_n(0)
+
(\mathscr Ru_n)(t)
\quad
\text{in }\mathcal H_{\nu}.
\end{align*}
We pass to the limit in $\mathcal H_{\nu}$, uniformly in $t$.  First, by \eqref{eq:K-Hnu-embedding},
\begin{align*}
\sup_{0\leq t\leq\Tend}
\norm{e^{-t\Lambda}\zeta_n(0)}_{\mathcal H_{\nu}}
\leq
\norm{\zeta_n(0)}_{\mathcal H_{\nu}}
\leq
\frac{C_P}{\beta^{1/2}}
\norm{\zeta_n(0)}_{\mathcal K}
\longrightarrow0.
\end{align*}
Secondly, \eqref{eq:causal-bound} gives
\begin{align*}
\sup_{0\leq t\leq\Tend}
\norm{(\mathscr Ru_n)(t)-(\mathscr Ru)(t)}_{\mathcal H_{\nu}}
\leq
M_0^{1/2}\,\Tend
\norm{u_n-u}_{C([0,\Tend];H)}
\longrightarrow0.
\end{align*}
Finally, $\zeta_n\to\zeta$ in $C([0,\Tend];\mathcal K)$, which is continuously embedded into $C([0,\Tend];\mathcal H_{\nu})$.  Hence \eqref{eq:modewise-mild} holds for every $t\in[0,\Tend]$.  The modewise restatement is the definition of equality in $\mathcal H_{\nu}$.
\end{proof}

\begin{corollary}[Memory potential of the mild solution]
\label{cor:graphwp}
Under the assumptions of \cref{thm:graphwp}, the abstract Cauchy problem has a unique mild solution
\begin{align*}
\displaystyle
X=(u,\zeta)\in \mathcal{C}([0,\Tend];\Wsp).
\end{align*}
Its memory potential
\begin{align*}
\displaystyle
\xi(t)
:=
J\zeta(t)
=
\int_{[0,\infty)}
\zeta(\lambda,t)\diff\nu(\lambda)
\end{align*}
belongs to
\begin{align*}
\displaystyle
\xi\in \mathcal{C}([0,\Tend];V)
\end{align*}
and satisfies
\begin{align}\label{eq:xi-graph-bound}
a_1(\xi(t),\xi(t))^{1/2}
\leq
M_0^{1/2}\norm{\zeta(t)}_{\mathcal K}.
\end{align}
By \cref{lem:modewise}(ii), for every $t\in[0,\Tend]$,
\begin{align*}
\zeta(\lambda,t)
=
\int_0^t
e^{-\lambda(t-s)}u(s)\diff s
\quad
\text{in }H
\quad
\text{for $\nu$-a.e. }\lambda.
\end{align*}
Integrating with respect to $\nu$ and interchanging the order of integration by Fubini's theorem, which is justified because $0\leq k\leq M_0$ and $u\in\mathcal C([0,\Tend];H)$, we obtain
\begin{align*}
\xi(t)
=
J\zeta(t)
=
\int_0^t k(t-s)u(s)\diff s
=
(k*u)(t),
\end{align*}
so that
\begin{align*}
\xi=k*u
\quad
\text{in }\mathcal C([0,\Tend];H).
\end{align*}
Uniqueness holds in $\mathcal{C}([0,\Tend];\Wsp)$ and does not require the additional assumption
\begin{align*}
\displaystyle
u\in L^2(0,\Tend;V).
\end{align*}
The stability estimate \eqref{eq:contraction} is independent of any positive coercivity constant for $a_0$.
\end{corollary}

\begin{corollary}[Hadamard well-posedness in the memory graph space]\label{cor:hadamard}
Under the assumptions of \cref{thm:graphwp}, and for zero initial memory, the solution pair belongs to $\mathfrak G_{\nu}(0,\Tend)$.  The data-to-solution map
\[
H\times L^{1}(0,\Tend;H)
\longrightarrow
\mathfrak G_{\nu}(0,\Tend),
\qquad
(u_{0},f)\longmapsto(u,\zeta),
\]
is Lipschitz continuous.  More precisely,
\begin{align}\label{eq:graph-hadamard}
\norm{(u,\zeta)}_{\mathfrak G_{\nu}(0,\Tend)}
\leq
\norm{u_{0}}_{H}
+
\norm{f}_{L^{1}(0,\Tend;H)},
\end{align}
and, for two data sets,
\begin{align*}
\norm{(u_{1}-u_{2},\zeta_{1}-\zeta_{2})}_{\mathfrak G_{\nu}(0,\Tend)}
\leq
\norm{u_{0,1}-u_{0,2}}_{H}
+
\norm{f_{1}-f_{2}}_{L^{1}(0,\Tend;H)}.
\end{align*}
Hence the degenerate problem is well posed in the sense of Hadamard in the memory graph space.
\end{corollary}

\begin{proof}
By \cref{lem:modewise}(ii), the mild solution satisfies $\zeta(t)=(\mathscr Ru)(t)$ in $\mathcal H_{\nu}$ for every $t\in[0,\Tend]$; hence $(u,\zeta)\in\mathfrak G_{\nu}(0,\Tend)$.  The estimate \eqref{eq:graph-hadamard} is the contraction estimate \eqref{eq:contraction} written in the norm \eqref{eq:trajectory-graph-norm}, and the Lipschitz estimate follows by applying \eqref{eq:contraction} to the difference of two mild solutions, exactly as in the proof of \cref{thm:graphwp}.
\end{proof}

\begin{remark}[Two complementary consequences of the memory structure]
\label{rem:twofaces}
\Cref{thm:structural-nogo} and \cref{thm:graphwp} describe two different aspects of the same memory mechanism. The decay
\begin{align*}
\displaystyle
m(\omega)\longrightarrow0
\quad
\text{as }|\omega|\to\infty
\end{align*}
prevents the memory dissipation from providing frequency-uniform $L^2(0,\Tend;V)$-coercivity. On the other hand, the internal-variable representation identifies the energy
\begin{align*}
\displaystyle
\norm{\zeta}_{\mathcal K}^{2}
=
\int_{[0,\infty)}
a_1(\zeta(\lambda),\zeta(\lambda))
\diff\nu(\lambda),
\end{align*}
and the coupling identity
\begin{align*}
\displaystyle
a_1(J\zeta,u)
=
\dual{\zeta}{Eu}_{\mathcal K}
\end{align*}
produces the exact cancellation in \eqref{eq:dissip}. Thus, the missing coercivity of $a_0$ is not replaced by the coercivity of the memory term in the instantaneous energy space. Instead, well-posedness is obtained in the extended memory space through the $m$-dissipativity of the augmented generator. Extended-state semigroup formulations for memory equations are classical
\cite{Dafermos1970,FabrizioGiorgiPata2010}. The point of
\cref{thm:graphwp} is that the contraction estimate uses only
\begin{align*}
\displaystyle
a_0(u,u)\geq0
\end{align*}
and contains no positive lower bound for $a_0$. It therefore remains valid for families of problems whose instantaneous coercivity constant tends to zero.
\end{remark}

\section{Weak formulation and identification}\label{sec:weak}
\subsection{The initial datum encoded in the weak formulation}
The basic graph-space solution need not belong to $L^2(0,\Tend;V)$. Therefore, neither the term $a_0(u,v)$ nor a time trace $u(0)$ should be used directly in the weak formulation. Both difficulties can be avoided by using the memory potential $\xi$ and the first-moment field $\eta$.

Recall that, whenever $0<M_0<\infty$,
\begin{align}\label{eq:xidot-recalled}
\partial_t\xi
=
M_0u-\eta,
\quad \text{or equivalently}\quad
M_0u=\partial_t\xi+\eta.
\end{align}
For a sufficiently regular solution, this identity gives
\begin{align*}
\int_0^\Tend a_0(u(t),v)\phi(t)\diff t
&=
\frac{1}{M_0}
\int_0^\Tend
a_0(\partial_t\xi(t)+\eta(t),v)\phi(t)\diff t
\\
&=
-\frac{1}{M_0}
\int_0^\Tend
a_0(\xi(t),v)\phi'(t)\diff t
+\frac{1}{M_0}
\int_0^\Tend
a_0(\eta(t),v)\phi(t)\diff t,
\end{align*}
provided that
\begin{align*}
\displaystyle
\xi(0)=0,
\quad
\phi(\Tend)=0.
\end{align*}
Thus, after integration by parts, the form $a_0$ acts only on $\xi$ and $\eta$, both of which are $V$-valued. Similarly, the time derivative of $u$ is interpreted distributionally:
\begin{align*}
\int_0^\Tend
(\partial_tu(t),v)_H\phi(t)\diff t
=
-\int_0^\Tend
(u(t),v)_H\phi'(t)\diff t
-
(u_0,v)_H\phi(0).
\end{align*}
The boundary term involving $u_0$ encodes the initial condition without requiring an a priori trace of $u$. These identities motivate the following definition.

\begin{definition}[Weak solution]\label{def:weak}
Assume \textnormal{(B1)--(B3)} together with
\begin{align*}
\displaystyle
0<M_0<\infty,
\quad
M_1<\infty.
\end{align*}
A function
\begin{align*}
\displaystyle
u\in L^\infty(0,\Tend;H)
\end{align*}
is called a \emph{weak solution} of \cref{eq:model} with data $(u_0,f)$ if the associated fields
\begin{align*}
\displaystyle
\xi(t)
=
\int_{[0,\infty)}
\zeta(\lambda,t)\diff\nu(\lambda),
\quad
\eta(t)
=
\int_{[0,\infty)}
\lambda\zeta(\lambda,t)\diff\nu(\lambda),
\end{align*}
where $\zeta$ is the internal-variable family \eqref{eq:zetafam}, whose defining Bochner integral is well defined for every $u\in L^\infty(0,\Tend;H)$, every $\lambda\geq0$, and every $t\in[0,\Tend]$,
satisfy
\begin{align}\label{eq:weak-regularity}
\xi
&\in
L^\infty(0,\Tend;V)
\cap
H^1(0,\Tend;H),
\\
\eta
&\in
L^2(0,\Tend;V),
\nonumber
\end{align}
together with
\begin{align}\label{eq:weak-xi-relation}
\xi(0)=0,
\quad
\partial_t\xi=M_0u-\eta
\quad\text{in }L^2(0,\Tend;H),
\end{align}
and if, for any $v\in V$ and any $\phi\in \mathcal{C}^\infty([0,\Tend])$ satisfying $\phi(\Tend)=0$,
\begin{align}\label{eq:weakform}
&-
\int_0^\Tend
(u(t),v)_H\phi'(t)\diff t
-
(u_0,v)_H\phi(0)
\nonumber\\
&\quad
-
\frac{1}{M_0}
\int_0^\Tend
a_0(\xi(t),v)\phi'(t)\diff t
+
\frac{1}{M_0}
\int_0^\Tend
a_0(\eta(t),v)\phi(t)\diff t
\nonumber\\
&\quad
+
\int_0^\Tend
a_1(\xi(t),v)\phi(t)\diff t
=
\int_0^\Tend
(f(t),v)_H\phi(t)\diff t.
\end{align}
\end{definition}
Every term in \cref{eq:weakform} is well defined under \cref{eq:weak-regularity}. Indeed, the terms involving $u$, $u_0$, and $f$ are evaluated in $H$, while the forms $a_0$ and $a_1$ act only on the $V$-valued fields $\xi$ and $\eta$. In particular, the formulation never requires $u(t)\in V$. The two boundary conditions
\begin{align*}
\displaystyle
\phi(\Tend)=0,
\quad
\xi(0)=0,
\end{align*}
remove the terminal and initial boundary terms arising from the integration by parts in $\xi$. The remaining boundary term
\begin{align*}
\displaystyle
-(u_0,v)_H\phi(0)
\end{align*}
encodes the initial condition for $u$. Thus, no time trace of $u$ is assumed in the definition itself.

\subsection{A priori estimate and identification of the weak solution}
We first derive the energy estimate for sufficiently regular solutions. The estimate exhibits the cancellation between the principal equation and the internal-variable family and does not use any positive lower bound for $a_0$.

\begin{theorem}[Coercivity-independent energy estimate]
\label{thm:apriori}
Assume \textnormal{(B1)} and \textnormal{(B3)}, and let $k$ be a locally integrable completely monotone kernel. Let $(u,\zeta)$ be a regular solution of the augmented system on $[0,\Tend]$ in the following sense: the principal equation may be tested with $u$, the internal-variable equation may be tested with $\zeta(\lambda,\cdot)$ for $\nu$-almost every $\lambda$, and all Hilbert-space chain rules, differentiations under the $\nu$-integral, and changes in the order of integration appearing in the proof below are justified. Assume that
\begin{align*}
\displaystyle
u(0)=u_0,\quad
\zeta(\lambda,0)=0
\quad\text{for $\nu$-almost every $\lambda$}.
\end{align*}
Then,
\begin{align}\label{eq:apriori}
&\sup_{t\in[0,\Tend]}
\left[
\norm{u(t)}_H^2
+
\int_{[0,\infty)}
a_1(\zeta(\lambda,t),\zeta(\lambda,t))
\diff\nu(\lambda)
\right]
\nonumber\\
&\quad
+
2\int_0^\Tend a_0(u(t),u(t))\diff t
\nonumber\\
&\quad
+
2\int_0^\Tend
\int_{[0,\infty)}
\lambda
a_1(\zeta(\lambda,t),\zeta(\lambda,t))
\diff\nu(\lambda)\diff t
\nonumber\\
&\leq
C_{\Tend}
\left(
\norm{u_0}_H^2
+
\norm{f}_{L^2(0,\Tend;H)}^2
\right),
\end{align}
where one may take
\begin{align}\label{eq:CT-def}
C_{\Tend}
:=
(2+\Tend)e^{\Tend}.
\end{align}
In particular, $C_{\Tend}$ does not involve any positive coercivity constant for $a_0$.
\end{theorem}

\begin{proof}
Set
\begin{align*}
\mathcal E(t)
&:=
\frac12\norm{u(t)}_H^2
+
\frac12\int_{[0,\infty)}
a_1(\zeta(\lambda,t),\zeta(\lambda,t))\diff\nu(\lambda),\\
\mathcal R(t)
&:=
a_0(u(t),u(t))
+
\int_{[0,\infty)}
\lambda a_1(\zeta(\lambda,t),\zeta(\lambda,t))\diff\nu(\lambda).
\end{align*}
Testing the principal equation with $u$, testing the internal-variable equation with
$\zeta(\lambda)$ in the $a_1$-inner product, and integrating with respect to $\nu$ give
\begin{align}\label{eq:energy-identity}
\frac{\diff}{\diff t}\mathcal E(t)+\mathcal R(t)=(f(t),u(t))_H.
\end{align}
The coupling terms cancel because
\begin{align*}
a_1(\xi(t),u(t))
=
\int_{[0,\infty)}a_1(u(t),\zeta(\lambda,t))\diff\nu(\lambda).
\end{align*}
Since $\mathcal R\geq0$, Young's inequality gives
\begin{align*}
\mathcal E'(t)
\leq
\mathcal E(t)+\frac12\norm{f(t)}_H^2.
\end{align*}
Hence Gr\"onwall's inequality and $\mathcal E(0)=\frac12\norm{u_0}_H^2$ yield
\begin{align*}
\sup_{t\in[0,\Tend]}\mathcal E(t)
\leq
\frac{e^{\Tend}}{2}
\left(\norm{u_0}_H^2+\norm{f}_{L^2(0,\Tend;H)}^2\right).
\end{align*}
Integrating \eqref{eq:energy-identity} and using the same estimate gives
\begin{align*}
\int_0^{\Tend}\mathcal R(t)\diff t
\leq
\frac{1+\Tend e^{\Tend}}{2}
\left(\norm{u_0}_H^2+\norm{f}_{L^2(0,\Tend;H)}^2\right).
\end{align*}
Combining these inequalities with
$e^{\Tend}+1+\Tend e^{\Tend}\leq(2+\Tend)e^{\Tend}$ proves
\cref{eq:apriori}. No positive lower bound for $a_0$ is used.
\end{proof}

\begin{theorem}[Identification of the graph-space solution]
\label{thm:exist}
Assume \textnormal{(B1)--(B3)} together with
\begin{align*}
\displaystyle
0<M_0<\infty,
\quad
M_1<\infty.
\end{align*}
Let
\begin{align*}
\displaystyle
X=(u,\zeta)\in \mathcal{C}([0,\Tend];\Wsp)
\end{align*}
be the unique mild solution given by \cref{thm:graphwp}, with
\begin{align*}
\displaystyle
X(0)=(u_0,0).
\end{align*}
We define
\begin{align*}
\displaystyle
\xi(t)
:=
\int_{[0,\infty)}
\zeta(\lambda,t)\diff\nu(\lambda),
\quad
\eta(t)
:=
\int_{[0,\infty)}
\lambda\zeta(\lambda,t)\diff\nu(\lambda).
\end{align*}
Then,
\begin{align*}
\displaystyle
u\in \mathcal{C}([0,\Tend];H),
\end{align*}
\begin{align*}
\displaystyle
\xi\in \mathcal{C}([0,\Tend];V)\cap H^1(0,\Tend;H),
\quad
\xi(0)=0,
\end{align*}
and
\begin{align*}
\displaystyle
\eta\in L^2(0,\Tend;V).
\end{align*}
Furthermore,
\begin{align*}
\displaystyle
\partial_t\xi=M_0u-\eta
\quad
\text{in }L^2(0,\Tend;H),
\end{align*}
and $u$ satisfies the weak formulation \cref{eq:weakform}. Consequently, the graph-space mild solution is a weak solution in the sense of \cref{def:weak}, and
\begin{align*}
\displaystyle
u(0)=u_0
\quad\text{strongly in }H.
\end{align*}
In addition,
\begin{align}\label{eq:reduced-apriori}
&\sup_{t\in[0,\Tend]}
\left[
\norm{u(t)}_H^2
+
\norm{\zeta(t)}_{\mathcal K}^2
\right]
+
2\int_0^\Tend
\int_{[0,\infty)}
\lambda
a_1(\zeta(\lambda,t),\zeta(\lambda,t))
\diff\nu(\lambda)\diff t
\nonumber\\
&\leq
C_{\Tend}
\left(
\norm{u_0}_H^2
+
\norm{f}_{L^2(0,\Tend;H)}^2
\right).
\end{align}
If, in addition, $u\in L^2(0,\Tend;V)$, then the full estimate
\cref{eq:apriori} holds.
\end{theorem}

\begin{proof}
We divide the proof into five steps.

\medskip
\noindent
\emph{Step 1: approximation by strong solutions.}
Because $\mathcal A$ is $m$-dissipative and densely defined, the mild solution can be approximated in $\mathcal{C}([0,\Tend];\Wsp)$ by strong solutions
\begin{align*}
\displaystyle
X_n=(u_n,\zeta_n)
\end{align*}
corresponding to data
\begin{align*}
\displaystyle
X_{0,n}\in D(\mathcal A),
\quad
f_n\in \mathcal{C}^1([0,\Tend];H),
\end{align*}
such that
\begin{align*}
\displaystyle
X_{0,n}\longrightarrow(u_0,0)
\quad\text{in }\Wsp,
\quad
f_n\longrightarrow f
\quad\text{in }L^2(0,\Tend;H).
\end{align*}
In particular,
\begin{align*}
\displaystyle
u_n\longrightarrow u
\quad\text{in }\mathcal{C}([0,\Tend];H),
\quad
\zeta_n\longrightarrow\zeta
\quad\text{in }\mathcal{C}([0,\Tend];\mathcal K).
\end{align*}

\medskip
\noindent
\emph{Step 2: the memory potential.}
The aggregation operator
\begin{align*}
\displaystyle
J\zeta
=
\int_{[0,\infty)}
\zeta(\lambda)\diff\nu(\lambda)
\end{align*}
is bounded from $\mathcal K$ into $V$. Therefore,
\begin{align*}
\displaystyle
\xi_n:=J\zeta_n
\longrightarrow
\xi:=J\zeta
\quad
\text{in }C([0,\Tend];V).
\end{align*}
Because the initial memory component converges to zero,
\begin{align*}
\displaystyle
\xi(0)=0.
\end{align*}
By \cref{lem:modewise}(ii), the internal component of the mild solution satisfies
\begin{align}\label{eq:mild-modewise-id}
\zeta(t)
=
(\mathscr Ru)(t)
\quad
\text{in }\mathcal H_{\nu}
\quad
\text{for every }t\in[0,\Tend].
\end{align}
Consequently, for each $t$, the family $\zeta(\cdot,t)$ coincides $\nu$-almost everywhere with the internal-variable family \eqref{eq:zetafam} generated by $u$; the fields $\xi$ and $\eta$ defined from the mild solution are therefore exactly the fields of \cref{def:weak} associated with $u$.  Integrating \eqref{eq:mild-modewise-id} with respect to $\nu$ and interchanging the order of integration by Fubini's theorem, as in \cref{cor:graphwp}, gives
\begin{align*}
\displaystyle
\xi(t)
=
\int_0^t k(t-s)u(s)\diff s
=
(k*u)(t).
\end{align*}

\medskip
\noindent
\emph{Step 3: construction of the first-moment field.}
We first record a memory-dissipation bound that is uniform in $n$. Each $X_n$ is a strong solution, so
\begin{align*}
\displaystyle
\frac{\diff}{\diff t}\frac12\norm{X_n(t)}_{\Wsp}^2
=
\dual{\mathcal AX_n(t)}{X_n(t)}_{\Wsp}
+
(f_n(t),u_n(t))_H
\end{align*}
on $[0,\Tend]$, and \eqref{eq:dissip-identity} expresses the first term on the right-hand side as the negative of the instantaneous and memory dissipation. The Gr\"onwall argument in the proof of \cref{thm:apriori} therefore applies verbatim, with the initial energy $\frac12\norm{X_{0,n}}_{\Wsp}^2$ in place of $\frac12\norm{u_0}_H^2$ (the hypothesis $\zeta(\lambda,0)=0$ enters that proof only through the value of the initial energy), and yields
\begin{align}\label{eq:diss-uniform}
\int_0^\Tend
\int_{[0,\infty)}
\lambda
a_1(\zeta_n(\lambda,t),\zeta_n(\lambda,t))
\diff\nu(\lambda)\diff t
\leq
D_*
:=
\frac{C_{\Tend}}{2}
\sup_{n}
\left(
\norm{X_{0,n}}_{\Wsp}^2
+
\norm{f_n}_{L^2(0,\Tend;H)}^2
\right),
\end{align}
where $D_*$ is finite because both approximating sequences converge.

The limit inherits this bound. For fixed $\Lambda>0$, we have $\lambda a_1\leq\Lambda a_1$ on $[0,\Lambda]$, so the convergence $\zeta_n\to\zeta$ in $\mathcal{C}([0,\Tend];\mathcal K)$ gives
\begin{align*}
\int_0^\Tend
\int_{[0,\Lambda]}
\lambda
a_1(\zeta(\lambda,t),\zeta(\lambda,t))
\diff\nu(\lambda)\diff t
=
\lim_{n\to\infty}
\int_0^\Tend
\int_{[0,\Lambda]}
\lambda
a_1(\zeta_n(\lambda,t),\zeta_n(\lambda,t))
\diff\nu(\lambda)\diff t
\leq
D_*.
\end{align*}
Letting $\Lambda\to\infty$, the monotone convergence theorem yields
\begin{align}\label{eq:diss-limit}
\int_0^\Tend
\int_{[0,\infty)}
\lambda
a_1(\zeta(\lambda,t),\zeta(\lambda,t))
\diff\nu(\lambda)\diff t
\leq
D_*
<\infty.
\end{align}
For almost every $t$, the Cauchy--Schwarz inequality with respect to the measure $\lambda\diff\nu(\lambda)$ yields
\begin{align}\label{eq:eta-CS}
a_1(\eta(t),\eta(t))
&=
a_1\left(
\int_{[0,\infty)}
\lambda\zeta(\lambda,t)\diff\nu(\lambda),
\int_{[0,\infty)}
\lambda\zeta(\lambda,t)\diff\nu(\lambda)
\right)
\nonumber\\
&\leq
M_1
\int_{[0,\infty)}
\lambda
a_1(\zeta(\lambda,t),\zeta(\lambda,t))
\diff\nu(\lambda).
\end{align}
Therefore,
\begin{align*}
\displaystyle
\eta\in L^2(0,\Tend;V);
\end{align*}
the same inequality applied to $\zeta_n$ together with \eqref{eq:diss-uniform} shows $\eta_n\in L^2(0,\Tend;V)$ for every $n$.

We next show the strong convergence
\begin{align*}
\displaystyle
\eta_n\longrightarrow\eta
\quad
\text{in }L^2(0,\Tend;V).
\end{align*}
For $\Lambda>0$, we split
\begin{align*}
\eta_n(t)-\eta(t)
&=
\int_{[0,\Lambda]}
\lambda\left(\zeta_n(\lambda,t)-\zeta(\lambda,t)\right)
\diff\nu(\lambda)
\\
&\quad
+
\int_{(\Lambda,\infty)}
\lambda\zeta_n(\lambda,t)
\diff\nu(\lambda)
-
\int_{(\Lambda,\infty)}
\lambda\zeta(\lambda,t)
\diff\nu(\lambda).
\end{align*}
For the first term, the Cauchy--Schwarz inequality with respect to the measure $\lambda\diff\nu(\lambda)$ restricted to $[0,\Lambda]$, whose total mass is at most $M_1$, gives
\begin{align*}
a_1\left(
\int_{[0,\Lambda]}\lambda(\zeta_n-\zeta)\diff\nu(\lambda),
\int_{[0,\Lambda]}\lambda(\zeta_n-\zeta)\diff\nu(\lambda)
\right)
\leq
M_1\Lambda
\norm{\zeta_n(t)-\zeta(t)}_{\mathcal K}^2,
\end{align*}
which tends to zero uniformly on $[0,\Tend]$ as $n\to\infty$. For the tail terms, the same inequality on $(\Lambda,\infty)$ gives, after integration in time and by \eqref{eq:diss-uniform},
\begin{align*}
\int_0^\Tend
a_1\left(
\int_{(\Lambda,\infty)}\lambda\zeta_n(\lambda,t)\diff\nu(\lambda),
\int_{(\Lambda,\infty)}\lambda\zeta_n(\lambda,t)\diff\nu(\lambda)
\right)
\diff t
\leq
\varepsilon(\Lambda)
D_*,
\end{align*}
where
\begin{align*}
\displaystyle
\varepsilon(\Lambda)
:=
\int_{(\Lambda,\infty)}
\lambda\diff\nu(\lambda),
\end{align*}
and likewise for the term involving $\zeta$ by \eqref{eq:diss-limit}. Because $M_1<\infty$, we have $\varepsilon(\Lambda)\to0$ as $\Lambda\to\infty$. Combining the three contributions through
\begin{align*}
\displaystyle
a_1(x+y+z,x+y+z)
\leq
3\left[a_1(x,x)+a_1(y,y)+a_1(z,z)\right],
\end{align*}
we obtain
\begin{align*}
\displaystyle
\limsup_{n\to\infty}
\int_0^\Tend
a_1\left(\eta_n(t)-\eta(t),\eta_n(t)-\eta(t)\right)
\diff t
\leq
6\varepsilon(\Lambda)D_*
\end{align*}
for every $\Lambda>0$. Therefore, the left-hand side vanishes, and the coercivity \textnormal{(B1)} gives the claimed strong convergence.

For the strong approximations,
\begin{align*}
\displaystyle
\partial_t\xi_n=M_0u_n-\eta_n.
\end{align*}
Using $u_n\to u$ in $\mathcal{C}([0,\Tend];H)$, the strong convergence $\eta_n\to\eta$ in $L^2(0,\Tend;V)$, and $\xi_n\to\xi$ in $\mathcal{C}([0,\Tend];V)$, we may pass to the limit to obtain
\begin{align*}
\displaystyle
\partial_t\xi=M_0u-\eta
\quad
\text{in }L^2(0,\Tend;H).
\end{align*}
It follows that
\begin{align*}
\displaystyle
\xi\in H^1(0,\Tend;H).
\end{align*}

\medskip
\noindent
\emph{Step 4: passage to the weak formulation.}
For every strong approximating solution, any $v\in V$, and any $\phi\in \mathcal{C}^\infty([0,\Tend])$ with $\phi(\Tend)=0$, testing the principal equation with $\phi v$ gives
\begin{align*}
&-
\int_0^\Tend
(u_n(t),v)_H\phi'(t)\diff t
-
(u_n(0),v)_H\phi(0)
\\
&\quad
+
\int_0^\Tend
a_0(u_n(t),v)\phi(t)\diff t
+
\int_0^\Tend
a_1(\xi_n(t),v)\phi(t)\diff t
\\
&=
\int_0^\Tend
(f_n(t),v)_H\phi(t)\diff t.
\end{align*}
Using
\begin{align*}
\displaystyle
u_n
=
\frac1{M_0}
\left(\partial_t\xi_n+\eta_n\right)
\end{align*}
in the $a_0$-term, we have
\begin{align*}
\int_0^\Tend a_0(u_n,v)\phi\diff t
&=
\frac1{M_0}
\int_0^\Tend
a_0(\partial_t\xi_n,v)\phi\diff t
+
\frac1{M_0}
\int_0^\Tend
a_0(\eta_n,v)\phi\diff t
\\
&=
-\frac1{M_0}
\int_0^\Tend
a_0(\xi_n,v)\phi'\diff t
+
\frac1{M_0}
\int_0^\Tend
a_0(\eta_n,v)\phi\diff t
+
r_n,
\end{align*}
where the boundary remainder is
\begin{align*}
\displaystyle
r_n
=
-\frac1{M_0}a_0(\xi_n(0),v)\phi(0).
\end{align*}
Because $\xi_n(0)\to0$ in $V$, one has $r_n\to0$. We pass to the limit. The convergences
\begin{align*}
\displaystyle
u_n\to u
\quad\text{in }\mathcal{C}([0,\Tend];H),
\quad
\xi_n\to\xi
\quad\text{in }\mathcal{C}([0,\Tend];V),
\end{align*}
\begin{align*}
\displaystyle
\eta_n\to\eta
\quad\text{in }L^2(0,\Tend;V),
\quad
f_n\to f
\quad\text{in }L^2(0,\Tend;H)
\end{align*}
give
\begin{align*}
&-
\int_0^\Tend
(u(t),v)_H\phi'(t)\diff t
-
(u_0,v)_H\phi(0)
\\
&\quad
-
\frac1{M_0}
\int_0^\Tend
a_0(\xi(t),v)\phi'(t)\diff t
+
\frac1{M_0}
\int_0^\Tend
a_0(\eta(t),v)\phi(t)\diff t
\\
&\quad
+
\int_0^\Tend
a_1(\xi(t),v)\phi(t)\diff t
=
\int_0^\Tend
(f(t),v)_H\phi(t)\diff t.
\end{align*}
This is precisely \cref{eq:weakform}.

\medskip
\noindent
\emph{Step 5: stability and the initial datum.}
Since
$X\in\mathcal C([0,\Tend];\Wsp)$, its first component satisfies
$u\in\mathcal C([0,\Tend];H)$ and $u(0)=u_0$. The energy identities for the
strong approximations, followed by weak lower semicontinuity in the weighted memory
space, give \cref{eq:reduced-apriori}. If, in addition,
$u\in L^2(0,\Tend;V)$, then the encoded weak formulation implies
$\partial_tu\in L^2(0,\Tend;V')$. The standard Hilbert-space chain rule and the
internal-variable energy identity therefore justify the calculation of
\cref{thm:apriori}, and the full estimate \cref{eq:apriori} follows.
\end{proof}

\begin{proposition}[Uniqueness in the energy subclass]
\label{prop:uniq}
Assume \textnormal{(B1)--(B3)} with
\begin{align*}
\displaystyle
0<M_0<\infty,
\quad
M_1<\infty.
\end{align*}
Then, there is at most one weak solution satisfying
\begin{align*}
\displaystyle
u\in L^2(0,\Tend;V).
\end{align*}
\end{proposition}

\begin{proof}
Let $u_1$ and $u_2$ be two such solutions with the same data, and set
\begin{align*}
\displaystyle
e:=u_1-u_2.
\end{align*}
Let $\zeta_e$ and $\xi_e$ denote the corresponding internal-variable family and memory potential. The difference has zero forcing and zero initial datum. Because
\begin{align*}
\displaystyle
e\in L^2(0,\Tend;V),
\quad
\xi_e\in L^\infty(0,\Tend;V),
\end{align*}
the equation implies
\begin{align*}
\displaystyle
\partial_te
=
-\mathsf A_0e-\mathsf A_1\xi_e
\in L^2(0,\Tend;V').
\end{align*}
Therefore,
\begin{align*}
\displaystyle
e\in \mathcal{C}([0,\Tend];H),
\quad
e(0)=0,
\end{align*}
and the standard $V'$-$V$ variational formulation holds for almost every $t$. We may therefore choose $v=e(t)$. Repeating the energy calculation of \cref{thm:apriori} gives, for every $t\in[0,\Tend]$,
\begin{align*}
&\frac12\norm{e(t)}_H^2
+
\frac12
\int_{[0,\infty)}
a_1(\zeta_e(\lambda,t),\zeta_e(\lambda,t))
\diff\nu(\lambda)
\\
&\quad
+
\int_0^t a_0(e(r),e(r))\diff r
\\
&\quad
+
\int_0^t
\int_{[0,\infty)}
\lambda
a_1(\zeta_e(\lambda,r),\zeta_e(\lambda,r))
\diff\nu(\lambda)\diff r
=
0.
\end{align*}
Every term on the left-hand side is non-negative. Consequently,
\begin{align*}
\displaystyle
\norm{e(t)}_H=0
\quad
\text{for every }t\in[0,\Tend],
\end{align*}
and hence $e\equiv0$.
\end{proof}

\begin{remark}[Exponential kernel]\label{rem:expspecial}
For
\begin{align*}
\displaystyle
k(t)=\gamma e^{-\gamma t},
\quad
\gamma>0,
\end{align*}
the representing measure is
\begin{align*}
\displaystyle
\nu=\gamma\delta_\gamma.
\end{align*}
Therefore,
\begin{align*}
\displaystyle
M_0=\gamma,
\quad
M_1=\gamma^2,
\quad
\norm{k}_{L^1(0,\infty)}=1.
\end{align*}
The internal-variable family reduces to a single variable
\begin{align*}
\displaystyle
\zeta(t):=\zeta(\gamma,t),
\end{align*}
and
\begin{align*}
\displaystyle
\xi=\gamma\zeta,
\quad
\eta=\gamma^2\zeta=\gamma\xi.
\end{align*}
\end{remark}

\begin{theorem}[Unconditional uniqueness for the exponential kernel]
\label{thm:uniqexp}
Let
\begin{align*}
\displaystyle
k(t)=\gamma e^{-\gamma t},
\quad
\gamma>0,
\end{align*}
and assume \textnormal{(B1)} and \textnormal{(B3)}. Then, the weak solution of \cref{def:weak} is unique without the additional assumption
\begin{align*}
\displaystyle
u\in L^2(0,\Tend;V).
\end{align*}
\end{theorem}

\begin{proof}
Let $u_1$ and $u_2$ be two weak solutions with the same data, and
set
\begin{align*}
\displaystyle
e:=u_1-u_2,
\quad
\xi_e:=k*e.
\end{align*}

\medskip
\noindent
\emph{Step 1: relations specific to the exponential kernel.}
Because
\begin{align*}
\displaystyle
M_0=\gamma,
\quad
\eta_e=\gamma\xi_e,
\end{align*}
the identity
\begin{align*}
\displaystyle
\partial_t\xi_e=M_0e-\eta_e
\end{align*}
becomes
\begin{align}\label{eq:exp-e-xi}
\partial_t\xi_e
=
\gamma e-\gamma\xi_e,
\quad
e
=
\frac1\gamma\partial_t\xi_e+\xi_e.
\end{align}
Furthermore,
\begin{align*}
\displaystyle
\xi_e\in L^\infty(0,\Tend;V),
\quad
\xi_e(0)=0.
\end{align*}
Because $e\in L^\infty(0,\Tend;H)$, relation
\eqref{eq:exp-e-xi} also gives
\begin{align*}
\displaystyle
\xi_e\in H^1(0,\Tend;H).
\end{align*}

\medskip
\noindent
\emph{Step 2: weak equation for the difference.}
Subtracting the two weak formulations gives
\begin{align}\label{eq:weakdiff}
&-
\int_0^\Tend
(e,\partial_t\Phi)_H\diff t
-
\frac1\gamma
\int_0^\Tend
a_0(\xi_e,\partial_t\Phi)\diff t
\nonumber\\
&\quad
+
\int_0^\Tend
a_0(\xi_e,\Phi)\diff t
+
\int_0^\Tend
a_1(\xi_e,\Phi)\diff t
=
0
\end{align}
for any
\begin{align*}
\displaystyle
\Phi\in H^1(0,\Tend;V),
\quad
\Phi(\Tend)=0.
\end{align*}
Finite sums of tensor-product functions $\sum_{j=1}^{N}\phi_j(t)v_j$, with $v_j\in V$, $\phi_j\in\mathcal C^\infty([0,\Tend])$, and $\phi_j(\Tend)=0$, are dense in
\begin{align*}
\left\{
\Phi\in H^1(0,\Tend;V):
\Phi(\Tend)=0
\right\}.
\end{align*}
Therefore, \eqref{eq:weakdiff}, first derived for product test functions, extends to every such $\Phi$.

\medskip
\noindent
\emph{Step 3: the integrated test function.}
We define
\begin{align*}
\displaystyle
G(t)
:=
\int_t^\Tend \xi_e(r)\diff r,
\quad
\Phi(t):=-G(t).
\end{align*}
Then,
\begin{align*}
\displaystyle
G'=-\xi_e,
\quad
G(\Tend)=0,
\quad
\partial_t\Phi=\xi_e.
\end{align*}
Because $\xi_e\in L^\infty(0,\Tend;V)$, this test function belongs to $H^1(0,\Tend;V)$.

\medskip
\noindent
\emph{Step 4: evaluation of the terms.}
Using \eqref{eq:exp-e-xi} and $\xi_e(0)=0$, we obtain
\begin{align*}
-\int_0^\Tend(e,\xi_e)_H\diff t
&=
-\frac1\gamma
\int_0^\Tend
(\partial_t\xi_e,\xi_e)_H\diff t
-
\int_0^\Tend\norm{\xi_e}_H^2\diff t
\\
&=
-\frac1{2\gamma}
\norm{\xi_e(\Tend)}_H^2
-
\int_0^\Tend\norm{\xi_e}_H^2\diff t.
\end{align*}
Furthermore,
\begin{align*}
\displaystyle
-\frac1\gamma
\int_0^\Tend
a_0(\xi_e,\xi_e)\diff t
\leq0.
\end{align*}
Because $\Phi=-G$ and $G'=-\xi_e$,
\begin{align*}
\int_0^\Tend a_0(\xi_e,\Phi)\diff t
&=
\int_0^\Tend a_0(G',G)\diff t
=
-\frac12a_0(G(0),G(0)),
\end{align*}
and similarly,
\begin{align*}
\displaystyle
\int_0^\Tend a_1(\xi_e,\Phi)\diff t
=
-\frac12a_1(G(0),G(0)).
\end{align*}

\medskip
\noindent
\emph{Step 5: conclusion.}
Substitution into \eqref{eq:weakdiff} gives
\begin{align*}
0
&=
-\frac1{2\gamma}\norm{\xi_e(\Tend)}_H^2
-
\int_0^\Tend\norm{\xi_e}_H^2\diff t
\\
&\quad
-
\frac1\gamma
\int_0^\Tend a_0(\xi_e,\xi_e)\diff t
-
\frac12a_0(G(0),G(0))
-
\frac12a_1(G(0),G(0)).
\end{align*}
Every term on the right-hand side is non-positive. Therefore, all of them vanish. In particular,
\begin{align*}
\displaystyle
\int_0^\Tend\norm{\xi_e(t)}_H^2\diff t=0,
\end{align*}
so
\begin{align*}
\displaystyle
\xi_e=0
\quad
\text{a.e. on }(0,\Tend).
\end{align*}
Relation \eqref{eq:exp-e-xi} then gives
\begin{align*}
\displaystyle
e=0.
\end{align*}
Therefore, the weak solution is unique.
\end{proof}

\begin{remark}[Why the integrated test is special to one relaxation time]
\label{rem:uniqopen}
The proof of \cref{thm:uniqexp} does not test the difference $e$ itself. Instead, it uses a time primitive of the $V$-valued memory potential $\xi_e$. This avoids assuming
\begin{align*}
\displaystyle
e\in L^2(0,\Tend;V).
\end{align*}
The argument closes for the exponential kernel because
\begin{align*}
\displaystyle
\eta_e=\gamma\xi_e.
\end{align*}
For a general completely monotone kernel,
\begin{align*}
\displaystyle
\xi_e
=
\int_{[0,\infty)}
\zeta_e(\lambda)\diff\nu(\lambda),
\quad
\eta_e
=
\int_{[0,\infty)}
\lambda\zeta_e(\lambda)\diff\nu(\lambda),
\end{align*}
and these two fields are not generally proportional. The resulting cross terms therefore have no definite sign. This is only a limitation of the integrated energy argument. \Cref{thm:graphwp} already gives unconditional uniqueness in the extended memory space whenever $M_0<\infty$. What remains outside the present graph-space theory is the infinite-mass case $M_0=\infty$.
\end{remark}

\begin{remark}[Fractional and infinite-mass kernels]
\label{rem:fracopen}
For the fractional kernel
\begin{align*}
\displaystyle
k(t)
=
\frac{t^{-\alpha}}{\Gamma(1-\alpha)},
\quad
0<\alpha<1,
\end{align*}
the representing measure is
\begin{align*}
\displaystyle
\diff\nu_\alpha(\lambda)
=
\frac{\sin(\pi\alpha)}{\pi}
\lambda^{\alpha-1}\diff\lambda.
\end{align*}
Its three relevant quantities are all infinite, but for different reasons. First,
\begin{align*}
\displaystyle
\norm{k}_{L^1(0,\infty)}
=
\int_{(0,\infty)}
\frac1\lambda\diff\nu_\alpha(\lambda)
=
\frac{\sin(\pi\alpha)}{\pi}
\int_0^\infty
\lambda^{\alpha-2}\diff\lambda
=
\infty.
\end{align*}
This divergence occurs at $\lambda=0$ and corresponds to the long-time tail generated by slow relaxation modes. By contrast,
\begin{align*}
\displaystyle
M_0
=
\int_{[0,\infty)}\diff\nu_\alpha
=
\infty,
\quad
M_1
=
\int_{[0,\infty)}
\lambda\diff\nu_\alpha
=
\infty
\end{align*}
because of divergence at $\lambda=\infty$. These fast relaxation modes correspond to the singular behaviour
\begin{align*}
\displaystyle
k(0^+)=\infty.
\end{align*}
The solvability of fractional memory equations is not itself open. Rather, the finite-moment construction developed here does not provide the present coercivity-independent graph-space certification when $M_0=M_1=\infty$. Indeed, the bounds
\begin{align*}
\displaystyle
a_1(\xi,\xi)
\leq
M_0
\int a_1(\zeta,\zeta)\diff\nu
\end{align*}
and
\begin{align*}
\displaystyle
a_1(\eta,\eta)
\leq
M_1
\int
\lambda a_1(\zeta,\zeta)\diff\nu
\end{align*}
become vacuous. A corresponding certified theory for infinite-mass kernels therefore requires a different state space and different weighted estimates.
\end{remark}

\begin{remark}[Reduction to a degenerate Kelvin--Voigt equation]
\label{rem:KV}
For the exponential kernel, the internal variable satisfies
\begin{align*}
\displaystyle
\partial_t\zeta+\gamma\zeta=u.
\end{align*}
Substituting
\begin{align*}
\displaystyle
u=\partial_t\zeta+\gamma\zeta
\end{align*}
into the principal equation gives
\begin{align*}
\displaystyle
\partial_{tt}\zeta
+
(\gamma+\mathsf A_0)\partial_t\zeta
+
\gamma(\mathsf A_0+\mathsf A_1)\zeta
=
f.
\end{align*}
Thus, the single-relaxation-time model is equivalent to a degenerate Kelvin--Voigt, or strongly damped wave, equation; for the spectral and stability theory of Kelvin--Voigt damping, see \cite{ChenLiuLiu1999,LiuZheng1999,LiuRao2006}. For a general completely monotone kernel, no single internal variable is available, and the corresponding reduced equation remains non-local in time.
\end{remark}

\section{Uniform graph-space stability}\label{sec:certified}
The no-go theorem excludes a frequency-uniform $L^2(0,\Tend;V)$-coercivity estimate supplied by the memory term. It does not, however, prevent stability in the extended memory space. The purpose of this subsection is to record that stability with explicit constants and to identify the corresponding target for a certified discretisation.

Although $a_0$ may still contribute non-negative dissipation, no positive lower bound for $a_0$ is used. The estimates below therefore remain valid for families of problems in which the instantaneous coercivity constant tends to zero.

We set
\begin{align}\label{eq:D0-def}
D_0
:=
\norm{u_0}_H^2
+
\norm{f}_{L^2(0,\Tend;H)}^2
\end{align}
and recall from \eqref{eq:CT-def} the explicit constant $C_{\Tend}=(2+\Tend)e^{\Tend}$ delivered by \cref{thm:apriori}.

\begin{corollary}[Explicit stability in the memory state space]
\label{cor:cert}
Assume \textnormal{(B1)--(B3)} together with
\begin{align*}
\displaystyle
0<M_0<\infty,
\quad
M_1<\infty,
\end{align*}
and let $X=(u,\zeta)$ be the solution given by \cref{thm:exist}. Then,
\begin{align}\label{eq:cert-graph}
&\sup_{t\in[0,\Tend]}
\left[
\norm{u(t)}_H^2
+
\int_{[0,\infty)}
a_1(\zeta(\lambda,t),\zeta(\lambda,t))
\diff\nu(\lambda)
\right]
\nonumber\\
&\quad
+
2\int_0^\Tend
\int_{[0,\infty)}
\lambda
a_1(\zeta(\lambda,t),\zeta(\lambda,t))
\diff\nu(\lambda)\diff t
\leq
C_{\Tend}D_0.
\end{align}
The memory potential and the first-moment field satisfy the more explicit bounds
\begin{align}\label{eq:cert-xi}
\sup_{t\in[0,\Tend]}
\seminorm{\xi(t)}_{V}^2
&\leq
\frac{M_0}{\beta}
e^{\Tend}D_0,
\\
\int_0^\Tend
\seminorm{\eta(t)}_{V}^2
\diff t
&\leq
\frac{M_1}{2\beta}
\left(1+\Tend e^{\Tend}\right)D_0.
\label{eq:cert-eta}
\end{align}
The same bounds hold for the full norm $\norm{\cdot}_{V}^{2}=\norm{\cdot}_{H}^{2}+\seminorm{\cdot}_{V}^{2}$ after multiplication of the right-hand sides by $1+C_{P}^{2}$, by the Poincar\'e inequality.
If, in addition,
\begin{align*}
\displaystyle
u\in L^2(0,\Tend;V),
\end{align*}
then the full energy estimate holds:
\begin{align}\label{eq:cert-full}
&\sup_{t\in[0,\Tend]}
\left[
\norm{u(t)}_H^2
+
\int_{[0,\infty)}
a_1(\zeta(\lambda,t),\zeta(\lambda,t))
\diff\nu(\lambda)
\right]
\nonumber\\
&\quad
+
2\int_0^\Tend a_0(u(t),u(t))\diff t
\nonumber\\
&\quad
+
2\int_0^\Tend
\int_{[0,\infty)}
\lambda
a_1(\zeta(\lambda,t),\zeta(\lambda,t))
\diff\nu(\lambda)\diff t
\leq
C_{\Tend}D_0.
\end{align}
None of these constants involves a positive coercivity constant for $a_0$.
\end{corollary}

\begin{proof}
For regular solutions, \cref{eq:energy-identity}, Young's inequality, and
Gr\"onwall's lemma give
\begin{align*}
\sup_{t\in[0,\Tend]}\mathcal E_0(t)
&\leq e^{\Tend}D_0,\\
2\int_0^{\Tend}\mathcal R_m(t)\diff t
&\leq(1+\Tend e^{\Tend})D_0,
\end{align*}
where
\begin{align*}
\mathcal E_0(t)
&:=\norm{u(t)}_H^2+\norm{\zeta(t)}_{\mathcal K}^2,\\
\mathcal R_m(t)
&:=\int_{[0,\infty)}
\lambda a_1(\zeta(\lambda,t),\zeta(\lambda,t))\diff\nu(\lambda).
\end{align*}
Combining these estimates proves \cref{eq:cert-full} for regular solutions.
The strong approximation used in \cref{thm:exist}, together with weak lower
semicontinuity, gives \cref{eq:cert-graph} for the graph-space solution.

Cauchy--Schwarz with respect to $\nu$ and $\lambda\diff\nu(\lambda)$ gives,
respectively,
\begin{align*}
a_1(\xi(t),\xi(t))
&\leq M_0\norm{\zeta(t)}_{\mathcal K}^2,\\
a_1(\eta(t),\eta(t))
&\leq M_1\mathcal R_m(t).
\end{align*}
The coercivity of $a_1$ now yields \cref{eq:cert-xi,eq:cert-eta}; the full-norm
bounds follow from Poincar\'e's inequality. Finally, if
$u\in L^2(0,\Tend;V)$, the energy identity is justified as in the last part of
\cref{thm:exist}, and \cref{eq:cert-full} follows.
\end{proof}

\section{The vanishing-instantaneous-coercivity limit}\label{sec:vanishing}
\label{subsec:vanishing-coercivity}
The estimates obtained above are uniform with respect to a positive coercivity constant for the instantaneous form. We strengthen this uniform stability into an actual singular-limit result. The memory kernel and its representing measure are kept fixed, while a coercive instantaneous contribution is allowed to vanish.

Let $a_{\mathrm v}:V\times V\to\R$ be a bounded symmetric coercive bilinear form. Thus, there exist constants $C_{\mathrm v},\beta_{\mathrm v}>0$ such that
\begin{align}
\lvert a_{\mathrm v}(v,w)\rvert
&\leq
C_{\mathrm v}\norm{v}_{V}\norm{w}_{V}
\quad
\text{for any }v,w\in V,
\label{eq:viscous-form-bounded}
\\
a_{\mathrm v}(v,v)
&\geq
\beta_{\mathrm v}\seminorm{v}_{V}^{2}
\quad
\text{for any }v\in V.
\label{eq:viscous-form-coercive}
\end{align}
Let $\mathsf B:V\to V'$ be the associated operator:
\begin{align*}
\dual{\mathsf Bv}{w}
=
a_{\mathrm v}(v,w).
\end{align*}
For $\varepsilon\in[0,1]$, we define
\begin{align}
a_0^\varepsilon(v,w)
&:=
a_0(v,w)
+
\varepsilon a_{\mathrm v}(v,w),
\label{eq:perturbed-form}
\\
\mathsf A_0^\varepsilon
&:=
\mathsf A_0+\varepsilon\mathsf B.
\label{eq:perturbed-operator}
\end{align}
For $\varepsilon>0$, the form $a_0^\varepsilon$ is coercive on $V$, whereas $a_0^0=a_0$ may be degenerate.

On the fixed memory graph space $\Wsp=H\times\mathcal K$, we define
\begin{align}
\mathcal A_\varepsilon(u,\zeta)
:=
\left(
-\mathsf A_0^\varepsilon u-\mathsf A_1J\zeta,\,
Eu-\Lambda\zeta
\right)
\label{eq:perturbed-generator}
\end{align}
with domain
\begin{align}
D(\mathcal A_\varepsilon)
:=
\left\{
(u,\zeta)\in V\times\mathcal K:
\mathsf A_0^\varepsilon u+\mathsf A_1J\zeta\in H,\ 
Eu-\Lambda\zeta\in\mathcal K
\right\}.
\label{eq:perturbed-generator-domain}
\end{align}
Thus, $\mathcal A_0=\mathcal A$.

\begin{proposition}[Uniform norm-resolvent convergence]
\label{prop:norm-resolvent}
Assume \textnormal{(B1)} and
\begin{align*}
0<M_0=\nu([0,\infty))<\infty.
\end{align*}
Then, for any $\varepsilon\in[0,1]$, the operator $\mathcal A_\varepsilon$ defined as \cref{eq:perturbed-generator,eq:perturbed-generator-domain} is densely defined and $m$-dissipative on $\Wsp$. Furthermore, there exists a constant $C_{\mathrm{res}}>0$, depending only on the forms $a_1$, $a_{\mathrm v}$, the measure $\nu$, and the embedding constants of $V\hookrightarrow H$, but independent of $\varepsilon$, such that
\begin{align}
\norm{
(I-\mathcal A_\varepsilon)^{-1}
-
(I-\mathcal A_0)^{-1}
}_{\mathcal L(\Wsp)}
\leq
C_{\mathrm{res}}\varepsilon
\quad
\text{for any }\varepsilon\in[0,1].
\label{eq:norm-resolvent-convergence}
\end{align}
In particular,
\begin{align}
(I-\mathcal A_\varepsilon)^{-1}Y
\longrightarrow
(I-\mathcal A_0)^{-1}Y
\quad
\text{in }\Wsp
\label{eq:strong-resolvent-convergence}
\end{align}
for any $Y\in\Wsp$.
\end{proposition}

\begin{proof}
For every $\varepsilon\in[0,1]$, the form
$a_0^\varepsilon=a_0+\varepsilon a_{\mathrm v}$ is bounded, symmetric, and
non-negative. Hence the proof of \cref{thm:graphwp}, with $a_0$ replaced by
$a_0^\varepsilon$, shows that $\mathcal A_\varepsilon$ is densely defined and
$m$-dissipative on the same space $\Wsp$.

Fix $Y=(F,G)\in\Wsp$ and write
\begin{align*}
(u_\varepsilon,\zeta_\varepsilon)
:=(I-\mathcal A_\varepsilon)^{-1}Y.
\end{align*}
As in the resolvent calculation in the proof of \cref{thm:graphwp},
\begin{align*}
\zeta_\varepsilon(\lambda)
&=\frac{u_\varepsilon+G(\lambda)}{1+\lambda},\\
(u_\varepsilon,v)_H+a_0(u_\varepsilon,v)
&+\varepsilon a_{\mathrm v}(u_\varepsilon,v)
+c_\nu a_1(u_\varepsilon,v)
=(F,v)_H-a_1(h_G,v),
\end{align*}
where
\begin{align*}
c_\nu:=\int_{[0,\infty)}\frac{1}{1+\lambda}\diff\nu(\lambda),
\qquad
h_G:=\int_{[0,\infty)}\frac{G(\lambda)}{1+\lambda}\diff\nu(\lambda).
\end{align*}
Let $C_1$ be a continuity constant for $a_1$ and set
\begin{align*}
\kappa_\nu:=\min\{1,c_\nu\beta\},
\qquad
C_Y:=(1+M_0C_1)^{1/2}.
\end{align*}
The Lax--Milgram estimate used in \cref{thm:graphwp} gives, uniformly in
$\varepsilon$,
\begin{align}\label{eq:res-u-uniform}
\norm{u_\varepsilon}_V
\leq
\frac{C_Y}{\kappa_\nu}\norm{Y}_{\Wsp}.
\end{align}
Let $(u_0,\zeta_0):=(I-\mathcal A_0)^{-1}Y$ and
$e_\varepsilon:=u_\varepsilon-u_0$. Subtracting the two reduced equations gives
\begin{align*}
(e_\varepsilon,v)_H+a_0(e_\varepsilon,v)+c_\nu a_1(e_\varepsilon,v)
=-\varepsilon a_{\mathrm v}(u_\varepsilon,v).
\end{align*}
Taking $v=e_\varepsilon$, using \eqref{eq:res-u-uniform}, and the continuity of
$a_{\mathrm v}$, we obtain
\begin{align*}
\norm{e_\varepsilon}_V
\leq
\varepsilon\frac{C_{\mathrm v}C_Y}{\kappa_\nu^2}\norm{Y}_{\Wsp}.
\end{align*}
Moreover,
\begin{align*}
\zeta_\varepsilon(\lambda)-\zeta_0(\lambda)
=\frac{e_\varepsilon}{1+\lambda},
\end{align*}
so
\begin{align*}
\norm{\zeta_\varepsilon-\zeta_0}_{\mathcal K}^2
\leq M_0C_1\norm{e_\varepsilon}_V^2.
\end{align*}
Consequently,
\begin{align*}
\norm{(I-\mathcal A_\varepsilon)^{-1}Y
-(I-\mathcal A_0)^{-1}Y}_{\Wsp}
\leq
\varepsilon
\frac{C_{\mathrm v}(1+M_0C_1)}{\kappa_\nu^2}
\norm{Y}_{\Wsp}.
\end{align*}
This proves \cref{eq:norm-resolvent-convergence} with
\begin{align*}
C_{\mathrm{res}}
:=\frac{C_{\mathrm v}(1+M_0C_1)}{\kappa_\nu^2}.
\end{align*}
\end{proof}

The passage from resolvent convergence to semigroup convergence is usually quoted from the Trotter--Kato theorem. In the present contractive setting, the passage is elementary, and carrying it out explicitly has a quantitative benefit: the semigroup difference inherits the $O(\varepsilon)$ resolvent rate of \cref{prop:norm-resolvent} on resolvent-smoothed initial states. For $\varepsilon\in[0,1]$, \cref{prop:norm-resolvent} and the Lumer--Phillips theorem provide the contraction semigroup $(S_\varepsilon(t))_{t\geq0}$ generated by $\mathcal A_\varepsilon$ on $\Wsp$, exactly as in \cref{thm:graphwp}. The following identity is the whole content of the direct argument.

\begin{lemma}[Semigroup difference through the resolvent difference]
\label{lem:semigroup-resolvent}
Assume \textnormal{(B1)} and
\begin{align*}
0<M_0=\nu([0,\infty))<\infty,
\end{align*}
and let $\varepsilon\in[0,1]$. Write
\begin{align*}
R_\varepsilon:=(I-\mathcal A_\varepsilon)^{-1},
\qquad
R_0:=(I-\mathcal A_0)^{-1}.
\end{align*}
Then, for every $y\in D(\mathcal A_0)$ and every $t\geq0$,
\begin{align}
R_\varepsilon S_0(t)y-S_\varepsilon(t)R_\varepsilon y
=
\int_0^t
S_\varepsilon(t-s)\,
(R_0-R_\varepsilon)\,
S_0(s)(I-\mathcal A_0)y
\diff s,
\label{eq:resolvent-commutator}
\end{align}
and consequently
\begin{align}
\norm{R_\varepsilon S_0(t)y-S_\varepsilon(t)R_\varepsilon y}_{\Wsp}
\leq
t\,
\norm{R_\varepsilon-R_0}_{\mathcal L(\Wsp)}\,
\norm{(I-\mathcal A_0)y}_{\Wsp}.
\label{eq:resolvent-commutator-bound}
\end{align}
\end{lemma}

\begin{proof}
We first record that the resolvents are contractions. For $X\in D(\mathcal A_\varepsilon)$, dissipativity gives
\begin{align*}
\norm{(I-\mathcal A_\varepsilon)X}_{\Wsp}\norm{X}_{\Wsp}
\geq
\dual{(I-\mathcal A_\varepsilon)X}{X}_{\Wsp}
=
\norm{X}_{\Wsp}^2-\dual{\mathcal A_\varepsilon X}{X}_{\Wsp}
\geq
\norm{X}_{\Wsp}^2,
\end{align*}
so $\norm{R_\varepsilon}_{\mathcal L(\Wsp)}\leq1$ for every $\varepsilon\in[0,1]$.

Fix $t>0$ and $y\in D(\mathcal A_0)$, and define
\begin{align*}
\psi(s):=S_\varepsilon(t-s)\,R_\varepsilon\,S_0(s)y,
\qquad
s\in[0,t].
\end{align*}
Because $y\in D(\mathcal A_0)$, the orbit $s\mapsto S_0(s)y$ is continuously differentiable in $\Wsp$ with derivative $S_0(s)\mathcal A_0y$. Hence
\begin{align*}
z(s):=R_\varepsilon S_0(s)y
\end{align*}
is continuously differentiable with $z'(s)=R_\varepsilon S_0(s)\mathcal A_0y$, takes values in $D(\mathcal A_\varepsilon)$, and
\begin{align*}
\mathcal A_\varepsilon z(s)=(R_\varepsilon-I)S_0(s)y
\end{align*}
is continuous in $s$, because $\mathcal A_\varepsilon R_\varepsilon=R_\varepsilon-I$ is bounded. For an admissible increment $h$, split
\begin{align*}
\psi(s+h)-\psi(s)
=
S_\varepsilon(t-s-h)\bigl(z(s+h)-z(s)\bigr)
+
\bigl(S_\varepsilon(t-s-h)-S_\varepsilon(t-s)\bigr)z(s).
\end{align*}
Dividing by $h$ and letting $h\to0$ from either side, the first term converges to $S_\varepsilon(t-s)z'(s)$ by the uniform contraction bound and strong continuity, and the second converges to $-S_\varepsilon(t-s)\mathcal A_\varepsilon z(s)$ because $z(s)\in D(\mathcal A_\varepsilon)$. Hence $\psi$ is differentiable on $[0,t]$ with continuous derivative
\begin{align*}
\psi'(s)
=
S_\varepsilon(t-s)
\bigl(
R_\varepsilon\mathcal A_0x-\mathcal A_\varepsilon R_\varepsilon x
\bigr),
\qquad
x:=S_0(s)y\in D(\mathcal A_0),
\end{align*}
where we also used $S_0(s)\mathcal A_0y=\mathcal A_0S_0(s)y$. Using $\mathcal A_\varepsilon R_\varepsilon=R_\varepsilon-I$ and $\mathcal A_0x=x-(I-\mathcal A_0)x$,
\begin{align*}
R_\varepsilon\mathcal A_0x-\mathcal A_\varepsilon R_\varepsilon x
&=
R_\varepsilon x-R_\varepsilon(I-\mathcal A_0)x
-R_\varepsilon x+x
\\
&=
x-R_\varepsilon(I-\mathcal A_0)x
=
(R_0-R_\varepsilon)(I-\mathcal A_0)x,
\end{align*}
where the last step uses $x=R_0(I-\mathcal A_0)x$. Since $(I-\mathcal A_0)S_0(s)y=S_0(s)(I-\mathcal A_0)y$, we obtain
\begin{align*}
\psi'(s)
=
S_\varepsilon(t-s)\,
(R_0-R_\varepsilon)\,
S_0(s)(I-\mathcal A_0)y.
\end{align*}
Integrating the continuous function $\psi'$ over $[0,t]$ and using
\begin{align*}
\psi(t)=R_\varepsilon S_0(t)y,
\qquad
\psi(0)=S_\varepsilon(t)R_\varepsilon y,
\end{align*}
we arrive at \cref{eq:resolvent-commutator}. The bound \cref{eq:resolvent-commutator-bound} follows because $S_\varepsilon(t-s)$ and $S_0(s)$ are contractions.
\end{proof}

\begin{theorem}[Vanishing instantaneous coercivity]
\label{thm:vanishing-coercivity}
Assume \textnormal{(B1)} and
\begin{align*}
0<M_0=\nu([0,\infty))<\infty.
\end{align*}
For $\varepsilon\in[0,1]$, let $(S_\varepsilon(t))_{t\geq0}$ be the contraction semigroup generated by $\mathcal A_\varepsilon$ on $\Wsp$. Then, for any $X_0\in\Wsp$ and any $\Tend>0$,
\begin{align}
\sup_{0\leq t\leq\Tend}
\norm{
S_\varepsilon(t)X_0-S_0(t)X_0
}_{\Wsp}
\longrightarrow0
\quad
\text{as }\varepsilon\downarrow0.
\label{eq:semigroup-convergence}
\end{align}
More generally, suppose that
\begin{align}
X_{0,\varepsilon}
&\longrightarrow
X_{0,0}
\quad
\text{in }\Wsp,
\label{eq:initial-data-convergence}
\\
f_\varepsilon
&\longrightarrow
f_0
\quad
\text{in }L^1(0,\Tend;H).
\label{eq:forcing-convergence}
\end{align}
Let
\begin{align}
X^\varepsilon(t)
:=
S_\varepsilon(t)X_{0,\varepsilon}
+
\int_0^t
S_\varepsilon(t-s)
(f_\varepsilon(s),0)
\diff s.
\label{eq:epsilon-mild}
\end{align}
Then,
\begin{align}
X^\varepsilon
\longrightarrow
X^0
\quad
\text{in } \mathcal{C}([0,\Tend];\Wsp).
\label{eq:mild-solution-convergence}
\end{align}
\end{theorem}

\begin{proof}
Throughout, write $R_\varepsilon:=(I-\mathcal A_\varepsilon)^{-1}$ and recall from \cref{prop:norm-resolvent} that
\begin{align*}
\norm{R_\varepsilon-R_0}_{\mathcal L(\Wsp)}
\leq
C_{\mathrm{res}}\,\varepsilon,
\qquad
\varepsilon\in[0,1].
\end{align*}
We divide the proof into three steps.

\emph{Step 1: resolvent-smoothed initial states.} Let $y\in D(\mathcal A_0)$ and consider the initial state $R_0y$. Because the resolvent of a generator commutes with the semigroup it generates,
\begin{align*}
S_0(t)R_0y=R_0S_0(t)y.
\end{align*}
Therefore,
\begin{align*}
S_\varepsilon(t)R_0y-S_0(t)R_0y
=
S_\varepsilon(t)(R_0-R_\varepsilon)y
+
\bigl(S_\varepsilon(t)R_\varepsilon y-R_\varepsilon S_0(t)y\bigr)
+
(R_\varepsilon-R_0)S_0(t)y.
\end{align*}
The first and third terms are each bounded in $\Wsp$ by $C_{\mathrm{res}}\,\varepsilon\norm{y}_{\Wsp}$, because $S_\varepsilon(t)$ and $S_0(t)$ are contractions. The middle term is bounded by \cref{eq:resolvent-commutator-bound}. Hence, for every $t\in[0,\Tend]$,
\begin{align}
\norm{S_\varepsilon(t)R_0y-S_0(t)R_0y}_{\Wsp}
\leq
C_{\mathrm{res}}\,\varepsilon
\Bigl(
2\norm{y}_{\Wsp}
+
\Tend\norm{(I-\mathcal A_0)y}_{\Wsp}
\Bigr).
\label{eq:smoothed-rate}
\end{align}
In particular, on initial states of the form $X_0=R_0y$ with $y\in D(\mathcal A_0)$, that is, on $D(\mathcal A_0^2)$, the semigroup convergence holds at the explicit rate $O(\varepsilon)$.

\emph{Step 2: density.} Let $X_0\in\Wsp$ and $\delta>0$. By the density of $D(\mathcal A_0)$ in $\Wsp$, established in Step~3 of the proof of \cref{thm:graphwp}, we may choose $z\in D(\mathcal A_0)$ with
\begin{align*}
\norm{X_0-z}_{\Wsp}<\delta,
\end{align*}
and then $y\in D(\mathcal A_0)$ with
\begin{align*}
\norm{y-(I-\mathcal A_0)z}_{\Wsp}<\delta.
\end{align*}
Since $\norm{R_0}_{\mathcal L(\Wsp)}\leq1$, as recorded in the proof of \cref{lem:semigroup-resolvent}, and $R_0(I-\mathcal A_0)z=z$,
\begin{align*}
\norm{R_0y-z}_{\Wsp}
=
\norm{R_0\bigl(y-(I-\mathcal A_0)z\bigr)}_{\Wsp}
<\delta,
\end{align*}
hence $\norm{X_0-R_0y}_{\Wsp}<2\delta$. Because all the semigroups are contractions, \cref{eq:smoothed-rate} gives
\begin{align*}
\sup_{0\leq t\leq\Tend}
\norm{S_\varepsilon(t)X_0-S_0(t)X_0}_{\Wsp}
\leq
4\delta
+
C_{\mathrm{res}}\,\varepsilon
\Bigl(
2\norm{y}_{\Wsp}
+
\Tend\norm{(I-\mathcal A_0)y}_{\Wsp}
\Bigr).
\end{align*}
Letting first $\varepsilon\downarrow0$ and then $\delta\downarrow0$ proves \cref{eq:semigroup-convergence}.

\emph{Step 3: the inhomogeneous problem.} Use the variation-of-constants formula. The terms
containing $X_{0,\varepsilon}-X_{0,0}$ and $f_\varepsilon-f_0$ tend to zero by the
contraction property. The remaining term is
\begin{align*}
\int_0^t
\bigl(S_\varepsilon(t-s)-S_0(t-s)\bigr)(f_0(s),0)\diff s.
\end{align*}
Its convergence to zero uniformly for $t\in[0,\Tend]$ follows first for
$\Wsp$-valued simple functions from \cref{eq:semigroup-convergence}, and
then for general $L^1$ data by density and the contraction bound. This proves
\cref{eq:mild-solution-convergence}.
\end{proof}

\begin{corollary}[A conditional quantitative convergence rate]
\label{cor:vanishing-rate}
Assume the hypotheses of \cref{thm:vanishing-coercivity}. Suppose that all problems have the same initial datum and forcing:
\begin{align}
X^\varepsilon(0)
=
X^0(0)
=
(u_0,0),
\quad
f_\varepsilon=f_0=f.
\label{eq:same-data-epsilon}
\end{align}
Assume, in addition, that the physical component of the limiting solution satisfies
\begin{align}
u^0\in L^2(0,\Tend;V).
\label{eq:limit-V-regularity}
\end{align}
Set
\begin{align*}
e^\varepsilon
:=
u^\varepsilon-u^0,
\quad
z^\varepsilon
:=
\zeta^\varepsilon-\zeta^0.
\end{align*}
Then, for any $t\in[0,\Tend]$,
\begin{align}
&
\norm{e^\varepsilon(t)}_H^2
+
\norm{z^\varepsilon(t)}_{\mathcal K}^2
+
2\int_0^t
a_0(e^\varepsilon(s),e^\varepsilon(s))
\diff s
+
\varepsilon
\int_0^t
a_{\mathrm v}
(e^\varepsilon(s),e^\varepsilon(s))
\diff s
\nonumber\\
&\quad
+
2\int_0^t
\int_{[0,\infty)}
\lambda
a_1
\left(
z^\varepsilon(\lambda,s),
z^\varepsilon(\lambda,s)
\right)
\diff\nu(\lambda)\diff s
\nonumber\\
&\leq
\varepsilon
\int_0^t
a_{\mathrm v}(u^0(s),u^0(s))
\diff s.
\label{eq:vanishing-pointwise-energy}
\end{align}
In particular,
\begin{align}
\sup_{t\in[0,\Tend]}
\left[
\norm{e^\varepsilon(t)}_H^2
+
\norm{z^\varepsilon(t)}_{\mathcal K}^2
\right]
\leq
\varepsilon
\int_0^\Tend
a_{\mathrm v}(u^0(t),u^0(t))
\diff t,
\label{eq:vanishing-energy-sup}
\end{align}
and
\begin{align}
\norm{
X^\varepsilon-X^0
}_{\mathcal C([0,\Tend];\Wsp)}
\leq
\varepsilon^{1/2}
\left(
\int_0^\Tend
a_{\mathrm v}(u^0(t),u^0(t))
\diff t
\right)^{1/2}.
\label{eq:vanishing-sqrt-rate}
\end{align}
Thus, the convergence rate in the memory graph norm is
$O(\varepsilon^{1/2})$.
\end{corollary}

\begin{proof}
We first justify the regularity needed for the energy calculation.

For any fixed $\varepsilon>0$, the perturbed instantaneous form satisfies
\begin{align*}
a_0^\varepsilon(v,v)
\geq
\varepsilon\beta_{\mathrm v}
\seminorm{v}_V^2.
\end{align*}
The strong-approximation argument used in the proof of \cref{thm:exist}, applied to the generator $\mathcal A_\varepsilon$, gives
\begin{align*}
\varepsilon
\int_0^\Tend
a_{\mathrm v}
(u^\varepsilon(t),u^\varepsilon(t))
\diff t
<\infty.
\end{align*}
Because $\varepsilon>0$ is fixed and $a_{\mathrm v}$ is coercive,
\begin{align}
u^\varepsilon
\in L^2(0,\Tend;V).
\label{eq:ueps-V-regularity}
\end{align}
By assumption,
\begin{align*}
u^0\in L^2(0,\Tend;V).
\end{align*}
Therefore,
\begin{align}
e^\varepsilon
=
u^\varepsilon-u^0
\in L^2(0,\Tend;V).
\label{eq:e-V-regularity}
\end{align}
Let
\begin{align*}
\xi^\varepsilon:=J\zeta^\varepsilon,
\quad
\xi^0:=J\zeta^0.
\end{align*}
The boundedness of $J:\mathcal K\to V$ gives
\begin{align*}
Jz^\varepsilon
=
\xi^\varepsilon-\xi^0
\in\mathcal C([0,\Tend];V).
\end{align*}
Subtracting the two principal equations, we obtain
\begin{align}
\partial_t e^\varepsilon
+
\mathsf A_0e^\varepsilon
+
\varepsilon\mathsf B e^\varepsilon
+
\mathsf A_1Jz^\varepsilon
=
-\varepsilon\mathsf Bu^0
\quad
\text{in }V'.
\label{eq:difference-principal}
\end{align}
All the terms on the right-hand side belong to $L^2(0,\Tend;V')$. Therefore,
\begin{align}
\partial_t e^\varepsilon
\in L^2(0,\Tend;V').
\label{eq:e-time-regularity}
\end{align}
Together with \eqref{eq:e-V-regularity}, this gives
\begin{align*}
e^\varepsilon
\in
L^2(0,\Tend;V)
\cap
H^1(0,\Tend;V')
\hookrightarrow
\mathcal C([0,\Tend];H),
\end{align*}
and the Hilbert-space chain rule yields
\begin{align}
\dual{
\partial_t e^\varepsilon
}{
e^\varepsilon
}_{V',V}
=
\frac12
\frac{\diff}{\diff t}
\norm{e^\varepsilon}_H^2
\quad
\text{in }\mathcal D'(0,\Tend).
\label{eq:e-chain-rule}
\end{align}
By \cref{lem:modewise}(ii), applied to $X^\varepsilon$ and to $X^0$, the internal-variable difference satisfies $z^\varepsilon(t)=(\mathscr Re^\varepsilon)(t)$ in $\mathcal H_{\nu}$ for every $t\in[0,\Tend]$; the proof of \cref{lem:modewise} applies verbatim to every $\mathcal A_\varepsilon$, $\varepsilon\in[0,1]$, because each $\mathcal A_\varepsilon$ is $m$-dissipative and densely defined and the internal-variable component of \cref{eq:perturbed-generator} does not depend on $\varepsilon$.  Because $e^\varepsilon\in L^2(0,\Tend;V)$ by \eqref{eq:e-V-regularity}, the causal representative
\begin{align*}
\widetilde z^\varepsilon(\lambda,t)
:=
\int_0^t
e^{-\lambda(t-s)}
e^\varepsilon(s)\diff s
\end{align*}
belongs, for every $\lambda\geq0$, to $H^1(0,\Tend;V)$, and, for every $t$, it defines the same element of $\mathcal K$ as $z^\varepsilon(t)$: the two families agree in $H$ for $\nu$-almost every $\lambda$, and $V\hookrightarrow H$ is injective.  We keep the notation $z^\varepsilon$ for this representative.  The internal-variable difference then satisfies, for every $\lambda\geq0$ and almost every $t\in(0,\Tend)$,
\begin{align}
\partial_t z^\varepsilon(\lambda,t)
+
\lambda z^\varepsilon(\lambda,t)
=
e^\varepsilon(t),
\quad
z^\varepsilon(\lambda,0)=0.
\label{eq:difference-internal}
\end{align}
Because $e^\varepsilon\in L^2(0,\Tend;V)$, this equation may be tested with $z^\varepsilon(\lambda,t)$ in the $a_1$-inner product. We obtain
\begin{align}
&
\frac12
\frac{\diff}{\diff t}
a_1
\left(
z^\varepsilon(\lambda,t),
z^\varepsilon(\lambda,t)
\right)
+
\lambda
a_1
\left(
z^\varepsilon(\lambda,t),
z^\varepsilon(\lambda,t)
\right)
=
a_1
\left(
e^\varepsilon(t),
z^\varepsilon(\lambda,t)
\right).
\label{eq:difference-internal-energy-lambda}
\end{align}
One may first integrate this identity over a bounded interval
$[0,\Lambda]$ in the relaxation variable. The graph-energy bounds
for $X^\varepsilon$ and $X^0$, together with
\begin{align*}
a_1(z^\varepsilon,z^\varepsilon)
\leq
2a_1(\zeta^\varepsilon,\zeta^\varepsilon)
+
2a_1(\zeta^0,\zeta^0),
\end{align*}
and the analogous inequality for the weighted dissipation, justify
the limit $\Lambda\to\infty$ by monotone and dominated
convergence. We therefore obtain
\begin{align}
&
\frac12
\frac{\diff}{\diff t}
\norm{z^\varepsilon(t)}_{\mathcal K}^2
+
\int_{[0,\infty)}
\lambda
a_1
\left(
z^\varepsilon(\lambda,t),
z^\varepsilon(\lambda,t)
\right)
\diff\nu(\lambda)
\nonumber\\
&\quad =
\int_{[0,\infty)}
a_1
\left(
e^\varepsilon(t),
z^\varepsilon(\lambda,t)
\right)
\diff\nu(\lambda).
\label{eq:difference-internal-energy}
\end{align}
We test \eqref{eq:difference-principal} with
$e^\varepsilon(t)$. Using
\eqref{eq:e-chain-rule}, we obtain
\begin{align}
&
\frac12
\frac{\diff}{\diff t}
\norm{e^\varepsilon(t)}_H^2
+
a_0(e^\varepsilon,e^\varepsilon)
+
\varepsilon
a_{\mathrm v}(e^\varepsilon,e^\varepsilon)
+
a_1(Jz^\varepsilon,e^\varepsilon)
=
-\varepsilon
a_{\mathrm v}(u^0,e^\varepsilon).
\label{eq:difference-principal-energy}
\end{align}
By the adjoint coupling identity and the symmetry of $a_1$,
\begin{align}
a_1(Jz^\varepsilon,e^\varepsilon)
=
\int_{[0,\infty)}
a_1
\left(
e^\varepsilon,
z^\varepsilon(\lambda)
\right)
\diff\nu(\lambda).
\label{eq:difference-cancellation}
\end{align}
Consequently, adding \eqref{eq:difference-internal-energy} to \eqref{eq:difference-principal-energy} cancels the coupling terms and gives
\begin{align}
&
\frac12
\frac{\diff}{\diff t}
\left[
\norm{e^\varepsilon(t)}_H^2
+
\norm{z^\varepsilon(t)}_{\mathcal K}^2
\right]
+
a_0(e^\varepsilon,e^\varepsilon)
+
\varepsilon
a_{\mathrm v}(e^\varepsilon,e^\varepsilon)
\nonumber\\
&\quad
+
\int_{[0,\infty)}
\lambda
a_1
\left(
z^\varepsilon(\lambda),
z^\varepsilon(\lambda)
\right)
\diff\nu(\lambda)
=
-\varepsilon
a_{\mathrm v}(u^0,e^\varepsilon).
\label{eq:difference-total-energy}
\end{align}
Because $a_{\mathrm v}$ is symmetric and positive definite, its Cauchy--Schwarz inequality gives
\begin{align*}
\lvert
a_{\mathrm v}(u^0,e^\varepsilon)
\rvert
\leq
a_{\mathrm v}(u^0,u^0)^{1/2}
a_{\mathrm v}(e^\varepsilon,e^\varepsilon)^{1/2}.
\end{align*}
Young's inequality therefore yields
\begin{align}
-\varepsilon
a_{\mathrm v}(u^0,e^\varepsilon)
&\leq
\varepsilon
\lvert
a_{\mathrm v}(u^0,e^\varepsilon)
\rvert
\leq
\frac{\varepsilon}{2}
a_{\mathrm v}(u^0,u^0)
+
\frac{\varepsilon}{2}
a_{\mathrm v}(e^\varepsilon,e^\varepsilon).
\label{eq:av-young}
\end{align}
Substitution into \eqref{eq:difference-total-energy} gives
\begin{align}
&
\frac{\diff}{\diff t}
\left[
\norm{e^\varepsilon(t)}_H^2
+
\norm{z^\varepsilon(t)}_{\mathcal K}^2
\right]
+
2a_0(e^\varepsilon,e^\varepsilon)
+
\varepsilon
a_{\mathrm v}(e^\varepsilon,e^\varepsilon)
\nonumber\\
&\quad
+
2
\int_{[0,\infty)}
\lambda
a_1
\left(
z^\varepsilon(\lambda),
z^\varepsilon(\lambda)
\right)
\diff\nu(\lambda)
\leq
\varepsilon
a_{\mathrm v}(u^0,u^0).
\label{eq:difference-energy-inequality}
\end{align}
The initial data are identical, and both internal-variable families initially vanish. Hence,
\begin{align*}
e^\varepsilon(0)=0,
\quad
z^\varepsilon(\lambda,0)=0
\quad
\text{for }\nu\text{-almost every }\lambda.
\end{align*}
Integrating \eqref{eq:difference-energy-inequality} from $0$ to $t$ gives precisely \eqref{eq:vanishing-pointwise-energy}. Dropping all non-negative integral terms from \eqref{eq:vanishing-pointwise-energy} gives
\begin{align*}
\norm{e^\varepsilon(t)}_H^2
+
\norm{z^\varepsilon(t)}_{\mathcal K}^2
\leq
\varepsilon
\int_0^t
a_{\mathrm v}(u^0(s),u^0(s))
\diff s.
\end{align*}
Taking the supremum over $t\in[0,\Tend]$ proves
\eqref{eq:vanishing-energy-sup}. Finally,
\begin{align*}
\norm{
X^\varepsilon(t)-X^0(t)
}_{\Wsp}^2
=
\norm{e^\varepsilon(t)}_H^2
+
\norm{z^\varepsilon(t)}_{\mathcal K}^2,
\end{align*}
so taking square roots gives
\eqref{eq:vanishing-sqrt-rate}.
\end{proof}

\begin{remark}[The instantaneous form varies, not the memory kernel]
\label{rem:operator-versus-kernel-limit}
The limit in \cref{thm:vanishing-coercivity} keeps the completely monotone kernel $k$ and its representing measure $\nu$ fixed. The singular parameter acts instead on the instantaneous spatial form:
\begin{align*}
a_0^\varepsilon
=
a_0+\varepsilon a_{\mathrm v}.
\end{align*}
It is therefore different from singular-memory limits in which the kernel itself changes or concentrates. The result also strengthens the uniform estimate of \cref{cor:cert}. That estimate shows that the graph-space stability constant does not deteriorate as the instantaneous coercivity vanishes, whereas \cref{thm:vanishing-coercivity} shows that the corresponding solutions actually converge in the fixed memory graph space to the degenerate solution generated by $\mathcal A_0$. The rate in \cref{cor:vanishing-rate} requires the additional regularity $u^0\in L^2(0,\Tend;V)$. This assumption is not automatic in the degenerate regime and should not be viewed as a consequence of the memory dissipation. It is precisely the lack of such frequency-uniform $V$-control that motivates the graph-space formulation.
\end{remark}

\section{Implications for structure-preserving and certified discretisation}\label{sec:discretisation}

\begin{remark}[Viscoelastic interpretation]
The degenerate setting is motivated, for example, by linearised Maxwell-type models of viscoelastic flow without solvent viscosity. In such models the instantaneous Newtonian dissipation is absent, and the relaxation mechanism is carried by the polymeric memory \cite{WangRenardy2011,RenardyHrusaNohel1987}. In the present abstract formulation, the loss of solvent viscosity corresponds to the loss of a positive lower bound for $a_0$, whereas the completely monotone kernel and the form $a_1$ represent the relaxation mechanism. The estimates above isolate the part of the stability that remains uniform when the instantaneous coercivity vanishes.
\end{remark}

\begin{remark}[Consequence of the no-go theorem for certification]
\label{rem:cert-impossible}
The estimate of \cref{cor:cert} is a stability estimate in the extended memory norm. It is not an $L^2(0,\Tend;V)$-coercivity estimate for $u$. This distinction is structural. By \cref{thm:structural-nogo}, the memory dissipation cannot provide a positive frequency-uniform constant $c$ such that
\begin{align*}
\displaystyle
\Dm[u](\Tend)
\geq
c\norm{u}_{L^2(0,\Tend;V)}^2
\end{align*}
for all admissible states. Consequently, a continuous or discrete certification argument whose reliability constant is obtained by dividing by such an instantaneous coercivity constant cannot remain uniform as that constant tends to zero. This does not rule out robust certification in the degenerate regime. It shows instead that the quantity to be certified must be the extended-memory stability measured in \cref{eq:cert-graph}, rather than an unavailable instantaneous $L^2(0,\Tend;V)$-coercivity bound.
\end{remark}

\begin{remark}[A structure-preserving discretisation programme]
\label{rem:cert-open}
The continuous analysis suggests three requirements for a robust discretisation of the degenerate problem.

\begin{enumerate}[leftmargin=2.2em,label=\textnormal{(\roman*)}]
\item
The discrete stability estimate should control the physical variable and the discrete internal variables in an analogue of the extended memory norm appearing in \cref{eq:cert-graph}.

\item
The discrete coupling between the principal equation and the internal-variable equations should preserve the adjoint relation that produces the continuous cancellation
\begin{align*}
\displaystyle
a_1(J\zeta,u)
=
\dual{\zeta}{Eu}_{\mathcal K}.
\end{align*}
Without this cancellation, the stability constant may depend on the missing coercivity of $a_0$.

\item
The reliability constant should remain bounded independently of the instantaneous coercivity parameter, the spatial mesh size, the time step, and, when the representing measure is approximated, the relaxation-spectrum quadrature.
\end{enumerate}

Thus, the natural discrete target is an estimate of the form
\begin{align*}
\displaystyle
\norm{X_h^n}_{\Wsp_h}^2
+
\text{discrete memory dissipation}
\leq
C_{\Tend}\,
\text{discrete data norm},
\end{align*}
with $C_{\Tend}$ independent of the instantaneous coercivity constant. Establishing such an estimate together with a computable and verifiable reliability bound is the certified-discretisation problem suggested by the present continuous theory.
\end{remark}

\section{Concluding remarks}
The principal conclusion is that loss of instantaneous coercivity does not destroy well-posed\-ness, but it changes the natural state space.  When the Bernstein representing measure has finite total mass $0<M_{0}<\infty$, the physical variable and the internal-variable family form an augmented state in $\Wsp=H\times\mathcal K$.  The aggregation and constant-embedding operators are adjoint in the memory energy, so the coupling terms cancel exactly and the augmented generator is $m$-dissipative.  The resulting contraction semigroup gives existence, uniqueness, and Lipschitz dependence on the data without any positive lower bound for $a_{0}$.  For zero prehistory, the trajectories lie in the memory graph space $\mathfrak G_{\nu}(0,\Tend)$ and satisfy the Hadamard estimate \cref{eq:graph-hadamard}.

Under the additional first-moment condition $M_{1}<\infty$, the memory potential $\xi$ and first-moment field $\eta$ provide an encoded weak formulation that never requires an a priori assumption $u\in L^{2}(0,\Tend;V)$.  The semigroup solution satisfies this formulation, attains the initial datum strongly in $H$, and obeys explicit graph-energy estimates.  Uniqueness holds unconditionally in the extended state space; within the larger encoded weak class it holds in the energy subclass, and for a single exponential kernel it holds without the additional $L^{2}(0,\Tend;V)$ assumption.

The same framework is stable under vanishing instantaneous coercivity.  For $a_{0}^{\varepsilon}=a_{0}+\varepsilon a_{v}$, the resolvents converge in operator norm at rate $O(\varepsilon)$ and the corresponding semigroups and inhomogeneous solutions converge in $C([0,\Tend];\Wsp)$.  If the limiting physical component belongs to $L^{2}(0,\Tend;V)$, the graph-norm convergence rate is $O(\varepsilon^{1/2})$.  These estimates identify the continuous quantity that a robust discretisation should preserve: the extended memory energy and its coupling cancellation, rather than an unavailable coercivity constant in the instantaneous energy norm.

The finite-mass restriction is structural for the present choice of state space.  Weakly singular fractional kernels have $M_{0}=M_{1}=\infty$, so the bounded aggregation and embedding estimates used here fail.  Their solvability is classical, but an equally explicit, coercivity-robust and computable stability theory in a suitable infinite-mass memory space remains open.

\paragraph{Funding.}
The author declares that no funds, grants, or other support were received during the preparation of this manuscript.

\paragraph{Data availability.}
No datasets were generated or analysed during the current study.

\paragraph{Competing interests.}
The author declares no competing interests.

\end{document}